\documentclass[11pt]{article}

\usepackage{amsfonts, amsmath, amssymb, amsthm}
\usepackage{mathtools, thmtools}
\usepackage{anysize}
\usepackage{hyperref, color, xcolor}
\usepackage{authblk}

\hypersetup
{
    colorlinks,
    linkcolor={blue!90!black},
    citecolor={blue!90!black},
    urlcolor={blue!90!black}
}

\title{New 2-closed groups that are not \\ automorphism groups of digraphs}
\author{John Bamberg, Michael Giudici, Jacob P. Smith}
\date{}
\affil{Centre for the Mathematics of Symmetry and Computation \\
Department of Mathematics and Statistics, \\
The University of Western Australia, \\
35 Stirling Highway, Perth, WA 6009, Australia. \\
john.bamberg@uwa.edu.au, michael.giudici@uwa.edu.au, jacob.smith@research.uwa.edu.au.
}

\newtheorem{theorem}{Theorem}[section]
\newtheorem{lemma}[theorem]{Lemma}

\newtheorem{proposition}[theorem]{Proposition}
\newtheorem{definition}[theorem]{Definition}

\newcommand{\AGL}{\mathrm{AGL}}
\newcommand{\Aut}{\mathrm{Aut}}
\newcommand{\Cay}{\mathrm{Cay}}
\renewcommand{\dim}{\mathrm{dim}}

\newcommand{\GF}{\mathrm{GF}}
\newcommand{\GL}{\mathrm{GL}}
\newcommand{\PG}{\mathrm{PG}}
\newcommand{\PGL}{\mathrm{PGL}}

\newcommand{\Sym}{\mathrm{Sym}}

\begin{document}

\maketitle

\begin{abstract}
In this paper we extend the construction of Giudici, Morgan and Zhou \cite{ref:giudici-primitive} to give the first known examples of nonregular, $2$-closed permutation groups of rank greater than $4$ that are not the automorphism group of any digraph. We also show that this construction only gives examples for four particular primes.

\end{abstract}

\section{Introduction}

The \emph{orbitals} of a permutation group $G$ acting on a set $\Omega$ are the induced orbits of $G$ on $\Omega \times \Omega$. In 1969, Wielandt \cite{ref:wielandt-invarient} introduced the \emph{$2$-closure} $G^{\left(2 \right)}$ of a permutation group $G$ on a set $\Omega$, defined as the largest permutation group on $\Omega$ that has the same orbitals as $G$. A permutation group is said to be \emph{$2$-closed} if it is equal to its $2$-closure. A permutation group is \emph{regular} if it is transitive and all its point stabilisers are trivial.

A \emph{digraph} $\Gamma$ is a structure consisting of a set $V{\left(\Gamma \right)}$ of vertices, along with a set of arcs $E \subseteq V^2$ not containing $\left(v, v \right)$ for any vertex $v$. A \emph{graph} is defined similarly, but with edges rather than arcs, where each edge is a $2$-subset of $V$. Automorphism groups of graphs and digraphs are always $2$-closed, but not every $2$-closed permutation group is the automorphism group of a digraph. In this paper we give new examples of such $2$-closed groups.

The classification of the finite regular permutation groups that are not the automorphism group of any graph was an area of active research throughout the 1960s and '70s, and was completed by Godsil \cite{ref:godsil-nonsolvable} in 1978. Babai \cite{ref:babai-drr} later showed that only five of these groups are not the automorphism group of any digraph.

Few papers have studied the \emph{nonregular} permutation groups that are not the automorphism group of any digraph. Ming-Yao Xu \cite{ref:xu-regular-subgroups} motivated the search for these groups in 2008. Defining $\mathcal{N}_2 \mathcal{R}$ to be the set of degrees of all the $2$-closed transitive permutation groups with no regular subgroups, and $\mathcal{NC}$ to be the set of all orders of vertex-transitive non-Cayley graphs, Xu claimed that in order to determine $\mathcal{N}_2 \mathcal{R} \setminus \mathcal{NC}$, ``we should first find nonregular 2-closed groups that are not the full automorphism groups of (di)graphs." Giudici, Morgan and Zhou \cite{ref:giudici-primitive} found the first examples of such groups in 2023. They gave three infinite families of nonregular groups that are not the automorphism group of any digraph, along with several isolated examples. The groups they found all have rank $4$, where the \emph{rank} of a permutation group is its number of orbitals. Moreover, they fully completed the classification of the finite, nonregular, primitive permutation groups of rank at most $4$, other than $1$-dimensional affine groups, that are not the automorphism group of any digraph.

Another motivation for the study of 2-closed groups that are not the automorphism group of a digraph is the Polycirculant Conjecture. Originally, Maru\v{s}i\v{c} \cite{ref:marusic-vertex} asked if the automorphism group of a vertex-transitive digraph always contains a semiregular permutation, that is, a nontrivial permutation with all cycles having the same length. This question was later extended by Klin \cite{ref:cameron-research} to the class of all vertex-transitive 2-closed groups and is now known as the Polycirculant Conjecture. See \cite{ref:cameron-transitive} and \cite{ref:arezoomand-polycirculant} for more information and recent results. Understanding the 2-closed groups that are not the automorphism group of a digraph yields information about the extent to which the Polycirculant Conjecture is more general than Maru\v{s}i\v{c}'s original question.

One of the infinite families given by Giudici, Morgan and Zhou \cite{ref:giudici-primitive} is the family containing $G{\left(m, 3 \right)}$, as defined in Definition \ref{def:G(m, q)}, for all integers $m \geqslant 2$. In this paper we show that this family can be generalised to groups of higher rank with the same properties, thus giving the first known nonregular permutation groups of rank greater than $4$ that are not the automorphism group of any digraph.

The following theorem is the main result of this paper.

\begin{theorem}
Let $m \geqslant 2$ be an integer and let $p$ be a prime. Then $G{\left(m, p \right)}$, as defined in Definition \ref{def:G(m, q)}, is a $2$-closed, nonregular permutation group that is not the automorphism group of any digraph if and only if $p \in \left\{3, 5, 7, 13 \right\}$.
\end{theorem}

This result follows from Theorems \ref{thm:q=5}, \ref{thm:q=7}, \ref{thm:q=13}, \ref{thm:2-closed} and \ref{thm:q-large-prime}. If $p$ is $5$, $7$ or $13$, the group $G{\left(m, q \right)}$ has rank $5$, $5$, or $7$ respectively. It remains open whether examples exist of arbitrarily large rank. We note here that all 2-closed groups constructed in this paper contain a regular subgroup, and hence a semiregular permutation, and so do not give counterexamples to the Polycirculant Conjecture.

\subsection*{Acknowledgements}
The second author was supported by Australian Research Council Discovery Project Grant DP190101024. The third author was supported by a Hackett Scholarship from The University of Western Australia.

\section{Preliminaries}

\subsection{Graph theory}

The \emph{Hamming graph} $H{\left(d, q \right)}$, where $d$ and $q$ are positive integers, is the graph whose vertices are the $d$-tuples of elements from the set $\left\{1, 2, \ldots, q \right\}$, and in which two vertices are adjacent if and only if they differ in exactly one coordinate. By \cite[Theorem 9.2.1]{ref:brouwer-distance}, its automorphism group is $S_q \wr S_d$.

Let $G$ be a group, and let $S$ be a nonempty subset of $G \setminus \left\{1_G \right\}$. The \emph{Cayley digraph} $\Cay{\left(G, S \right)}$ is the digraph whose vertices are the elements of $G$, and in which each pair $\left(g, h \right) \in G \times G$ is an arc if and only if $g h^{-1} \in S$. The following is a simple consequence of \cite[Lemma 3.7.3]{ref:godsil-algebraic}.

\begin{lemma} \label{lm:endomorphism-automorphism}
Suppose $V$ is a vector space with a subset $S$, and let $C$ be the Cayley digraph $C = \Cay{\left(V,\, S \right)}$, where $V$ is considered as an additive group. Then any invertible linear transformation of $V$ is an automorphism of $C$ if and only if it preserves $S$.
\end{lemma}

If $G$ is a permutation group acting on a set $\Omega$, then each orbital $B$ has a corresponding \emph{orbital digraph}, defined as the digraph with vertex set $\Omega$ and arc set $B$.

\subsection{Cross-ratios}

Consider a line $\ell$ in the projective space $\PG(n, F)$. If we choose \emph{reference vectors} $v, w \in F^{n+1}$ such that $\ell$ is the $2$-space generated by $v$ and $w$, then every vector in $\ell$ has the form $\lambda_1 v + \lambda_2 w$ for some $\lambda_1, \lambda_2 \in F$. Therefore all the $1$-spaces (or projective points) contained inside $\ell$ can be expressed as either $\left<w \right>$ or $\left<v + \lambda w \right>$ for some unique $\lambda \in F$. This means that if we define $P_{\left(\lambda \right)} = \left<v + \lambda w \right>$ for all $\lambda \in F$, and $P_{\left(\infty \right)} = \left<w \right>$, then each point $P$ on the line $\ell$ can be expressed uniquely in the form $P_{\left(\lambda \right)}$, where $\lambda \in F \cup \left\{\infty \right\}$.

Using the above notation, as well as the convention that division by zero yields infinity and vice versa, the \emph{cross-ratio} of four points $P_{\left(A \right)}, P_{\left(B \right)}, P_{\left(C \right)}, P_{\left(D \right)}$ on the line $\ell$ is defined as follows:
\begin{displaymath}
\mathcal{R} {\left(P_{\left(A \right)}, P_{\left(B \right)}; P_{\left(C \right)}, P_{\left(D \right)}\right)} = \frac{\left(C-A \right) \left(D-B \right)}{\left(C-B \right) \left(D-A \right)} \in F \cup \left\{\infty \right\}.
\end{displaymath}

Although this defines the cross-ratio in terms of the arbitrarily chosen reference vectors $v$ and $w$, it can be shown that the cross-ratio of any four collinear points does not depend on the choice of reference vectors (see for example \cite[Section 3.5, Theorem 1]{ref:garner-projective}).

It is well known that if $\theta \in \PGL{\left(n+1,\, F \right)}$, then $\theta$ preserves the cross-ratio of any four collinear points $P, Q, R, S \in \PG{\left(n, F \right)}$. That is to say, $\mathcal{R}{\left(P^\theta, Q^\theta; R^\theta, S^\theta \right)} = \mathcal{R}{\left(P, Q; R, S \right)}$. The following result \cite[Lemma 4.4]{ref:bamberg-analytic} illustrates how the cross-ratio $\mathcal{R} {\left(P, Q; R, S \right)}$ of any four points is affected by a reordering of the points.

\begin{lemma} \label{lm:permute-cross-ratio}
Let $P, Q, R, S$ be four collinear points in $\PG{\left(n, F \right)}$, and let $\sigma \in \Sym{\left(\left\{P, Q, R, S \right\}\right)}$ be a permutation of these points. Let $r = \mathcal{R}{\left(P, Q; R, S \right)}$, and let $r'$ be the cross-ratio of the points' images under $\sigma$, that is, $r' = \mathcal{R}{\left(P^\sigma, Q^\sigma; R^\sigma, S^\sigma \right)}$. Then the relation between $r$ and $r'$ is given in Table \ref{tbl:permuted-cross-ratios}.
\end{lemma}

\begin{table}
\renewcommand*{\arraystretch}{1.6}
\centering
\begin{tabular}{cc}
\hline
$\sigma$ & $r'$ \\
\hline
$\left(\right)$, ${\left(PQ \right)} {\left(RS \right)}$, ${\left(PR \right)} {\left(QS \right)}$, ${\left(PS \right)} {\left(QR \right)}$ & $r$ \\
$\left(PR \right)$, $\left(QS \right)$, $\left(PQRS \right)$, $\left(PSRQ \right)$ & $\frac{r}{r-1}$ \\
$\left(PS \right)$, $\left(QR \right)$, $\left(PQSR \right)$, $\left(PRSQ \right)$ & $1-r$ \\
$\left(PQ \right)$, $\left(RS \right)$, $\left(PRQS \right)$, $\left(PSQR \right)$ & $\frac{1}{r}$ \\
$\left(PQS \right)$, $\left(PRQ \right)$, $\left(PSR \right)$, $\left(QRS \right)$ & $\frac{1}{1-r}$ \\
$\left(PQR \right)$, $\left(PRS \right)$, $\left(PSQ \right)$, $\left(QSR \right)$ & $\frac{r-1}{r}$ \\
\hline
\end{tabular}
\caption{The value of $r'$ in terms of $r$ for each $\sigma \in \Sym{\left(\left\{P, Q, R, S \right\}\right)}$, given $r = \mathcal{R}{\left(P, Q; R, S \right)}$ and $r' = \mathcal{R}{\left(P^\sigma, Q^\sigma; R^\sigma, S^\sigma \right)}$.}
\label{tbl:permuted-cross-ratios}
\end{table}

\subsection{Tensor product decompositions}

Suppose $V$ and $W$ are $n$- and $m$-dimensional vector spaces over a field $F$, with bases $\left\{v_1, v_2, \ldots, v_n \right\}$ and $\left\{w_1, w_2, \ldots, w_m \right\}$ respectively. If $A \in \GL{\left(V \right)}$ and $B \in \GL{\left(W \right)}$, denote by $A \circ B$ the linear transformation of $V \otimes W$ mapping $v_i \otimes w_j$ to $v_i^A \otimes w_j^B$ for all $1 \leqslant i \leqslant n$ and $1 \leqslant i \leqslant m$. Denote by $\GL{\left(V \right)} \circ \GL{\left(W \right)}$ the group of all such linear transformations.

The following is from Lemma 4.4.5 of \cite{ref:xu-aschbacher}.

\begin{lemma} \label{lm:preserves-simple-tensors-preserves-decomp}
Let $X = V \otimes W$. If $\theta \in \GL{\left(X \right)}$ preserves the set of all simple tensors, then $\theta$ preserves the tensor product decomposition $V \otimes W$.
\end{lemma}

\section{The group $G{\left(m, q \right)}$}

Here we define the main family of groups discussed in this paper.

\begin{definition} \label{def:G(m, q)}
Let $F$ be the field $\GF{\left(q \right)}$ where $q$ is an odd prime power, and let $V$ and $W$ be $2$- and $m$-dimensional vector spaces over $F$, respectively, with $m \geqslant 2$. Let $\left\{e_1, e_2 \right\}$ be a basis for $V$, and let $\left\{f_1, f_2, \ldots, f_m \right\}$ be a basis for $W$. Now define $D_8 \leqslant \GL{\left(2, q \right)}$ to be the following group of linear transformations of $V$, expressed with respect to the basis $\left\{e_1, e_2 \right\}$:
\begin{align}
D_8 = {\left<{\begin{bmatrix} 1 & 0 \\ 0 & -1 \end{bmatrix}},\ \begin{bmatrix} 0 & 1 \\ 1 & 0 \end{bmatrix} \right>}. \label{eqn:D_8}
\end{align}

Finally, the group $G{\left(m, q \right)}$ is defined as follows.
\begin{displaymath}
G{\left(m, q \right)} = \left(V \otimes W \right) \rtimes \left(D_8 \circ \GL{\left(m, q \right)}\right) \leqslant \AGL{\left(2m, q \right)}.
\end{displaymath}
\end{definition}

\subsection{Suborbits} \label{sc:suborbits}

Let $V$, $W$ and $F$ be as defined in Definition \ref{def:G(m, q)}, and for convenience let $G = G {\left(m, q \right)}$. We will refer to the orbits of the point stabiliser $G_0 = D_8 \circ \GL{\left(m, q \right)}$ as the \emph{suborbits} of $G$. The following lemma will be of use in determining the suborbits.

\begin{lemma} \label{lm:V-basis-under-D8}
$\left(e_1, e_2 \right)^{D_8} = \left\{{\left(e_1,\, \pm e_2 \right)},\ {\left(-e_1,\, \pm e_2 \right)},\ {\left(e_2,\, \pm e_1 \right)},\ \left(-e_2,\, \pm e_1 \right)\right\}$.
\end{lemma}
\begin{proof}
This follows from the definition of $D_8$ in Equation \eqref{eqn:D_8}.
\end{proof}

Clearly, the set $\delta = \left\{0 \right\} \subseteq V \otimes W$ is an orbit of $G_0$ on $V \otimes W$. Let $\Delta_A$ be the orbit of $G_0$ containing the vector $e_1 \otimes f_1$. Then
\begin{displaymath}
\Delta_A = \left(e_1 \otimes f_1 \right)^{G_0} = {\left\{e_1^a \otimes f_1^g \mid a \in D_8,\ g \in \GL {\left(m, q \right)}\right\}}.
\end{displaymath}
It can be seen from Lemma~\ref{lm:V-basis-under-D8} that $e_1^{D_8} = \left\{\pm e_1, \pm e_2 \right\}$, hence we have
\begin{align*}
\Delta_A &= \left\{\pm e_1 \otimes w,\ \pm e_2 \otimes w \mid w \in W \setminus \left\{0 \right\}\right\} \\
&= \big(\left(\langle e_1 \rangle \otimes W \right) \cup \left(\langle e_2 \rangle \otimes W \right)\big) \setminus {\left\{0 \right\}}.
\end{align*}

Let $\Delta_B$ be the orbit of $G_0$ containing the vector $e_1 \otimes f_1 + e_2 \otimes f_2$. Then
\begin{displaymath}
\Delta_B = \left(e_1 \otimes f_1 + e_2 \otimes f_2 \right)^{G_0} = {\left\{e_1^a \otimes f_1^g + e_2^a \otimes f_2^g \mid a \in D_8,\ g \in \GL{\left(m, q \right)}\right\}}.
\end{displaymath}
Using Lemma~\ref{lm:V-basis-under-D8} and the fact that $\GL{\left(m, q \right)}$ acts transitively on the set of pairs of linearly independent vectors, we find that
\begin{align*}
\Delta_B &= {\left\{e_1 \otimes w_1 \pm e_2 \otimes w_2,\ -e_1 \otimes w_1 \pm e_2 \otimes w_2 \mid w_1, w_2 \in W \setminus {\left\{0 \right\}},\ \dim \left<w_1, w_2 \right> = 2 \right\}} \\
&= {\left\{e_1 \otimes w_1 + e_2 \otimes w_2 \mid w_1, w_2 \in W \setminus {\left\{0 \right\}},\ \dim \left<w_1, w_2 \right> = 2 \right\}}.
\end{align*}
That is, $\Delta_B$ is the set of all non-simple tensors in $V \otimes W$.

The remaining orbits of $G_0$ will each be denoted by $\Delta_\lambda$ for some $\lambda \in F \setminus \left\{0 \right\}$ using the following definition:
\begin{align} \label{eqn:delta_lambda}
\Delta_\lambda = {\left\{\left(e_1 \pm \lambda e_2 \right) \otimes w,\ \left(e_1 \pm \lambda^{-1} e_2 \right) \otimes w \mid w \in W \setminus \left\{0 \right\}\right\}}.
\end{align}
To see that each $\Delta_\lambda$ is a (not necessarily distinct) orbit of $G_0$, let $\lambda \in F \setminus \left\{0 \right\}$, and let $\Delta$ be the orbit containing the vector $\left(e_1 + \lambda e_2 \right) \otimes f_1$. Then
\begin{align*}
\Delta &= \left(\left(e_1 + \lambda e_2 \right) \otimes f_1 \right)^{G_0} \\
&= \left\{(e_1 + \lambda e_2)^a \otimes f_1^g \mid a \in D_8,\ g \in \GL{\left(m, q \right)}\right\} \\
&= {\left\{(e_1^a + \lambda e_2^a) \otimes f_1^g \mid a \in D_8,\ g \in \GL{\left(m, q \right)}\right\}}.
\end{align*}
Using Lemma~\ref{lm:V-basis-under-D8}, we have
\begin{align*}
\Delta &= \big\{{\left(e_1 \pm \lambda e_2 \right)} \otimes w,\ \left(-e_1 \pm \lambda e_2 \right) \otimes w, \\
&\ \ \ \ \ \ \ \left(e_2 \pm \lambda e_1 \right) \otimes w,\ \left(-e_2 \pm \lambda e_1 \right) \otimes w \mid w \in W \setminus {\left\{0 \right\}}\big\} \\
&= \big\{{\left(e_1 \pm \lambda e_2 \right)} \otimes w,\ \left(e_1 \pm \lambda^{-1} e_2 \right) \otimes w \mid w \in W \setminus {\left\{0 \right\}}\big\} \\
&= \Delta_\lambda.
\end{align*}

To summarise, the suborbits of $G$ are $\delta$, $\Delta_A$, $\Delta_B$, and $\Delta_\lambda$ for each $\lambda \in F \setminus \left\{0 \right\}$, noting that different values of $\lambda$ may correspond to the same suborbit (see the following lemma).

\begin{lemma} \label{lm:suborbit-equivalence}
If $\lambda, \mu \in F \setminus \left\{0 \right\}$, then $\Delta_\lambda = \Delta_\mu$ if and only if $\mu \in \left\{\pm \lambda,\ \pm \lambda^{-1} \right\}$.
\end{lemma}
\begin{proof}
This follows immediately from the definition given in Equation \eqref{eqn:delta_lambda} above.
\end{proof}

\subsection{Orbital digraphs}

Denote the non-trivial orbitals of $G{\left(m, q \right)}$ by $B_A$, $B_B$ and $B_\lambda$ such that they correspond to the suborbits $\Delta_A$, $\Delta_B$ and $\Delta_\lambda$ respectively for each $\lambda \in F \setminus \left\{0 \right\}$. Denote by $\Gamma_A$, $\Gamma_B$ and $\Gamma_\lambda$ the corresponding orbital digraphs. Since $G{\left(m, q \right)}$ contains the regular normal subgroup $V \otimes W$, the orbital digraph corresponding to any suborbit $\Delta$ is the Cayley digraph $\Cay{\left(V \otimes W,\ \Delta \right)}$.

To check whether $G{\left(m, q \right)}$ is the automorphism group of any digraph, it suffices to check only the unions of the non-trivial orbital digraphs. Some of these orbital digraph unions are isomorphic to the Hamming graph $H{\left(2, q^m \right)}$; to see this, we first note \cite[Lemma 3.1]{ref:giudici-primitive} which is as follows.

\begin{lemma} \label{lm:direct-sum-hamming}
Suppose $X$ and $Y$ are $m$-dimensional subspaces of $V \otimes W$ such that $V \otimes W = X \oplus Y$, and let $S = \left(X \cup Y \right) \setminus \left\{0 \right\}$. Then the Cayley digraph $\Cay{\left(V \otimes W,\ S \right)}$ is isomorphic to the Hamming graph $H{\left(2, q^m \right)}$.
\end{lemma}

Using this result, we can prove the following two lemmas.

\begin{lemma} \label{lm:Gamma_A,B}
The group $G = G{\left(m, q \right)}$ is not the automorphism group of its orbital digraph $\Gamma_A$, nor of $\Gamma_B$.
\end{lemma}
\begin{proof}
If we let $X = \langle e_1 \rangle \otimes W$ and $Y = \langle e_2 \rangle \otimes W$, then $X$ and $Y$ are $m$-dimensional subspaces of $V \otimes W$ that intersect only at $0$, so $V \otimes W = X \oplus Y$. Since $\Gamma_A = \Cay{\left(V \otimes W,\ \Delta_A \right)}$ and $\Delta_A = X \cup Y \setminus \{0\}$, it follows from Lemma~\ref{lm:direct-sum-hamming} that $\Gamma_A$ is isomorphic to the Hamming graph $H{\left(2, q^m \right)}$. Since this has automorphism group $S_{q^m} \wr S_2$, it is clear that $\Aut{\left(\Gamma_A \right)}$ is much larger than $G$.

Recall that $\Delta_B$ is the set of all non-simple tensors in $V \otimes W$, and is therefore invariant under $\GL{\left(2, q \right)} \circ \GL{\left(m, q \right)}$. Hence $G_0 < \GL{\left(2, q \right)} \circ \GL{\left(m, q \right)} \leqslant \Aut{\left(\Gamma_B \right)}_0$ by Lemma~\ref{lm:endomorphism-automorphism}, so $G$ is not the full automorphism group of $\Gamma_B$.
\end{proof}

\begin{lemma} \label{lm:Gamma_1,i}
The group $G{\left(m, q \right)}$ is not the automorphism group of its orbital digraph $\Gamma_1$, nor of $\Gamma_i$ for any $i \in \GF{\left(q \right)}$ such that $i^2 = -1$.
\end{lemma}
\begin{proof}
Let $X = \langle e_1 + e_2 \rangle \otimes W$, and let $Y = \langle e_1 - e_2 \rangle \otimes W$. Then $X$ and $Y$ are $m$-dimensional subspaces of $V \otimes W$ that intersect only at $0$, so $V \otimes W = X \oplus Y$.

Recalling the definition of $\Delta_1$, we see that
\begin{displaymath}
\Delta_1 = \left\{ \left(e_1 \pm e_2 \right) \otimes w \mid w \in W \setminus \left\{0 \right\} \right\} = X \cup Y \setminus {\left\{0 \right\}},
\end{displaymath}
so by Lemma~\ref{lm:direct-sum-hamming} we have $\Gamma_1 = \Cay{\left(V \otimes W,\ \Delta_1 \right)} \cong H{\left(2, q^m \right)}$.

Now if $i^2 = -1$, then $i^{-1} = -i$. It follows that
\begin{displaymath}
\Delta_i = {\left\{\left(e_1 \pm i e_2 \right) \otimes w \mid w \in W \setminus \left\{0 \right\}\right\}}.
\end{displaymath}
Using the same argument as above, but with $X = \left<e_1 + i e_2 \right> \otimes W$ and $Y = \left<e_1 - i e_2 \right> \otimes W$, we see that $\Gamma_i$ is also isomorphic to $H{\left(2, q^m \right)}$. The automorphism group of $H{\left(2, q^m \right)}$ is $S_{q^m} \wr S_2$, which is much larger than $G$, so we conclude that $G$ is not the automorphism group of $\Gamma_1$ nor~$\Gamma_i$.
\end{proof}

\begin{proposition} \label{prop:primitive}
The group $G{\left(m, q \right)}$ acts primitively on $V \otimes W$.
\end{proposition}
\begin{proof}
By \cite[1.12]{ref:higman-intersection}, it suffices to show that all the nontrivial orbital digraphs are connected. We have seen in the proof of Lemma~\ref{lm:Gamma_A,B} that $\Gamma_A$ is isomorphic to the Hamming graph $H{\left(2, q^m \right)}$ and is therefore connected. Recall that $\Gamma_B$ is the set of tensors in $V \otimes W$ that are not simple, and observe that any nonzero simple tensor $v \otimes w \in V \otimes W$ is the sum of two simple tensors via
\begin{displaymath}
v \otimes w = \left(\frac{1}{2} v \otimes w + v' \otimes w' \right) + \left(\frac{1}{2} v \otimes w - v' \otimes w' \right)
\end{displaymath}
for some $v' \in V$ and $w' \in W$ such that $\dim \left<v, v' \right> = \dim \left<w, w' \right> = 2$. Hence $\Delta_B$ is a generating set for the group $V \otimes W$, so $\Gamma_B = \Cay{\left(V \otimes W,\ \Delta_B \right)}$ is connected.

Now consider the orbital digraph $\Gamma_\lambda = \Cay{\left(V \otimes W,\ \Delta_\lambda \right)}$ where $\lambda \in \GF{\left(q \right)} \setminus \left\{0 \right\}$. Let $X = \left<e_1 + \lambda e_2 \right> \otimes W$, let $Y = \left<e_1 - \lambda e_2 \right> \otimes W$, and let $\delta_\lambda = \left(X \cup Y \right) \setminus \left\{0 \right\}$. Observe then that $\delta_\lambda$ is a subset of $\Delta_\lambda$, so it suffices to show that $\Cay{\left(V \otimes W,\ \delta_\lambda \right)}$ is connected since it is a subgraph of $\Gamma_\lambda$. Since $X$ and $Y$ are $m$-dimensional subspaces of $V \otimes W$ that intersect only at $0$, we have that $V \otimes W = X \oplus Y$. Hence by Lemma~\ref{lm:direct-sum-hamming}, $\Cay{\left(V \otimes W,\ \delta_\lambda \right)}$ is isomorphic to the Hamming graph $H{\left(2, q^m \right)}$, and is therefore connected as required.
\end{proof}

\subsection{Automorphisms of the orbital digraphs}
In this section we take $z \in \left\{4, 6 \right\}$ and $I = \left\{1, 2, \ldots, z \right\}$. For each $i \in I$ we define $\mu_i \in \GF{\left(q \right)}$, such that each $\mu_i$ is distinct.

\begin{definition} \label{def:projection-coordinates}
For each odd $i \in I$, consider the following direct sum decomposition of $V \otimes W$:
\begin{displaymath}
    V \otimes W = \left(\left<e_1 + \mu_i e_2 \right> \otimes W \right) \oplus {\left(\left<e_1 + \mu_{i+1} e_2 \right> \otimes W \right)}.
\end{displaymath}
Then for each odd $i \in I$ and each $x \in V \otimes W$, define $\pi_i{\left(x \right)}$ and $\pi_{i+1}{\left(x \right)}$ to be the unique vectors in $W$ such that
\begin{displaymath}
    x = \Big({\left(e_1 + \mu_i e_2 \right)} \otimes \pi_i{\left(x \right)} \Big) + \Big({\left(e_1 + \mu_{i+1} e_2 \right)} \otimes \pi_{i+1}{\left(x \right)}\Big).
\end{displaymath}
\end{definition}

\begin{lemma} \label{lm:projections-preserve-addition-and-multiplication}
Let $x, y \in V \otimes W$, let $\kappa \in \GF{\left(q \right)}$ and let $i \in I$. Then $\pi_i{\left(x + y \right)} = \pi_i{\left(x \right)} + \pi_i{\left(y \right)}$ and $\pi_i{\left(\kappa x \right)} = \kappa \pi_i{\left(x \right)}$.
\end{lemma}
\begin{proof}
We have
\begin{align*}
    x + y &= \left(e_1 + \mu_1 e_2 \right) \otimes \pi_1{\left(x \right)} + \left(e_1 + \mu_2 e_2 \right) \otimes \pi_2{\left(x \right)} + \left(e_1 + \mu_1 e_2 \right) \otimes \pi_1{\left(y \right)} + \left(e_1 + \mu_2 e_2 \right) \otimes \pi_2{\left(y \right)} \\
    &= \left(e_1 + \mu_1 e_2 \right) \otimes \left(\pi_1{\left(x \right)} + \pi_1{\left(y \right)}\right) + \left(e_1 + \mu_2 e_2 \right) \otimes \left(\pi_2{\left(x \right)} + \pi_2{\left(y \right)}\right)
\end{align*}
and
\begin{align*}
    \kappa x &= \kappa \left(\left(e_1 + \mu_1 e_2 \right) \otimes \pi_1{\left(x \right)} + \left(e_1 + \mu_2 e_2 \right) \otimes \pi_2{\left(x \right)}\right) \\
    &= \left(e_1 + \mu_1 e_2 \right) \otimes \kappa \pi_1{\left(x \right)} + \left(e_1 + \mu_2 e_2 \right) \otimes \kappa \pi_2{\left(x \right)},
\end{align*}
so we see the lemma holds for $i \in \left\{1, 2 \right\}$. The proof is similar for $i \in \left\{3, 4, 5, 6 \right\}$.
\end{proof}

\begin{lemma} \label{lm:projection-relations}
Let $i, j, k \in I$ such that $i \neq j$. Then there exist $\kappa_1, \kappa_2 \in \GF{\left(q \right)}$ such that
\begin{displaymath}
    \pi_k{\left(x \right)} = \kappa_1 \pi_i{\left(x \right)} + \kappa_2 \pi_j{\left(x \right)}
\end{displaymath}
holds for all $x \in V \otimes W$. Moreover, if $i, j, k$ are all distinct then $\kappa_1 \neq 0$ and $\kappa_2 \neq 0$.
\end{lemma}
\begin{proof}
This is clearly true if $k=i$ or $k=j$, so assume that $i, j, k$ are all distinct.

First, suppose $z=4$. It follows from Definition \ref{def:projection-coordinates} that
\begin{align*}
    &e_1 \otimes \left(\pi_1{\left(x \right)} + \pi_2{\left(x \right)}\right) + e_2 \otimes \left(\mu_1 \pi_1{\left(x \right)} + \mu_2 \pi_2{\left(x \right)}\right) \\
    =\ &e_1 \otimes \left(\pi_3{\left(x \right)} + \pi_4{\left(x \right)}\right) + e_2 \otimes \left(\mu_3 \pi_3{\left(x \right)} + \mu_4 \pi_4{\left(x \right)}\right)
\end{align*}
for all $x \in V \otimes W$, and hence
\begin{displaymath}
    e_1 \otimes \left(\pi_1{\left(x \right)} + \pi_2{\left(x \right)} - \pi_3{\left(x \right)} - \pi_4{\left(x \right)}\right) + e_2 \otimes \left(\mu_1 \pi_1{\left(x \right)} + \mu_2 \pi_2{\left(x \right)} - \mu_3 \pi_3{\left(x \right)} + \mu_4 \pi_4{\left(x \right)}\right) = 0
\end{displaymath}
for all $x \in V \otimes W$.
This gives us
\begin{align}
    \pi_1{\left(x \right)} + \pi_2{\left(x \right)} - \pi_3{\left(x \right)} - \pi_4{\left(x \right)} &= 0, \label{eqn:projection-relation-unsolved-1} \\
    \mu_1 \pi_1{\left(x \right)} + \mu_2 \pi_2{\left(x \right)} - \mu_3 \pi_3{\left(x \right)} - \mu_4 \pi_4{\left(x \right)} &= 0 \label{eqn:projection-relation-unsolved-2}
\end{align}
for all $x \in V \otimes W$, from which we can solve for $\pi_k{\left(x \right)}$ in terms of $\pi_i{\left(x \right)}$ and $\pi_j{\left(x \right)}$. For example, if $k=1$ we have
\begin{align*}
    \pi_1{\left(x \right)} &= \frac{\mu_4 - \mu_2}{\mu_1 - \mu_4} \pi_2{\left(x \right)} + \frac{\mu_3 - \mu_4}{\mu_1 - \mu_4} \pi_3{\left(x \right)}, \\
    \pi_1{\left(x \right)} &= \frac{\mu_3 - \mu_2}{\mu_1 - \mu_3} \pi_2{\left(x \right)} + \frac{\mu_4 - \mu_3}{\mu_1 - \mu_3} \pi_4{\left(x \right)}, \\
    \pi_1{\left(x \right)} &= \frac{\mu_3 - \mu_2}{\mu_1 - \mu_2} \pi_3{\left(x \right)} + \frac{\mu_4 - \mu_2}{\mu_1 - \mu_2} \pi_4{\left(x \right)}
\end{align*}
for all $x \in V \otimes W$. Since $\mu_1$, $\mu_2$, $\mu_3$, $\mu_4$ are distinct, the above coefficients are defined and non-zero. It is simple to verify that this also holds for $k \in \left\{2, 3, 4 \right\}$.

If $z=6$, then we have the following in addition to Equations \eqref{eqn:projection-relation-unsolved-1} and \eqref{eqn:projection-relation-unsolved-2}:
\begin{align}
    \pi_1{\left(x \right)} + \pi_2{\left(x \right)} - \pi_5{\left(x \right)} - \pi_6{\left(x \right)} &= 0, \label{eqn:projection-relation-unsolved-3} \\
    \mu_1 \pi_1{\left(x \right)} + \mu_2 \pi_2{\left(x \right)} - \mu_5 \pi_5{\left(x \right)} - \mu_6 \pi_6{\left(x \right)} &= 0 \label{eqn:projection-relation-unsolved-4}
\end{align}
for all $x \in V \otimes W$. We can then use Equations \eqref{eqn:projection-relation-unsolved-1}, \eqref{eqn:projection-relation-unsolved-2}, \eqref{eqn:projection-relation-unsolved-3} and \eqref{eqn:projection-relation-unsolved-4} to solve for $\pi_k{\left(x \right)}$ in terms of $\pi_i{\left(x \right)}$ and $\pi_j{\left(x \right)}$ similarly to the case where $z=4$.
\end{proof}

\begin{lemma} \label{lm:two-projections-determines-tensor}
Given any $i, j \in I$ with $i \neq j$, and any $w, w' \in W$, there exists a unique $x \in V \otimes W$ such that $\pi_i{\left(x \right)} = w$ and $\pi_j{\left(x \right)} = w'$.
\end{lemma}
\begin{proof}
From Lemma~\ref{lm:projection-relations} there exist some $\kappa_1, \kappa_1', \kappa_2, \kappa_2' \in \GF{\left(q \right)}$ such that
\begin{align}
    \pi_1{\left(x \right)} &= \kappa_1 \pi_i{\left(x \right)} + \kappa_1' \pi_j{\left(x \right)}, \label{eqn:pi_1-relation} \\
    \pi_2{\left(x \right)} &= \kappa_2 \pi_i{\left(x \right)} + \kappa_2' \pi_j{\left(x \right)} \label{eqn:pi_2-relation}
\end{align}
for all $x \in V \otimes W$. Therefore letting
\begin{displaymath}
    x = \left(e_1 + \mu_1 e_2 \right) \otimes \left(\kappa_1 w + \kappa_1' w' \right) + \left(e_1 + \mu_2 e_2 \right) \otimes {\left(\kappa_2 w + \kappa_2' w' \right)},
\end{displaymath}
we know that Equations \eqref{eqn:pi_1-relation} and \eqref{eqn:pi_2-relation} must hold. Note that if $\kappa_1 \kappa_2' = \kappa_1' \kappa_2$, then solving these equations simultaneously for $\pi_i{\left(x \right)}$ and $\pi_j{\left(x \right)}$ would not be possible, contradicting Lemma~\ref{lm:projection-relations}. Solving the equations simultaneously therefore yields
\begin{align*}
    \pi_i{\left(x \right)} &= \frac{\kappa_2' \pi_1{\left(x \right)} - \kappa_1' \pi_2{\left(x \right)}}{\kappa_1 \kappa_2' - \kappa_1' \kappa_2} \\
    &= \frac{\kappa_2' \left(\kappa_1 w + \kappa_1' w' \right) - \kappa_1' \left(\kappa_2 w + \kappa_2' w' \right)}{\kappa_1 \kappa_2' - \kappa_1' \kappa_2} \\
    &= w, \\
    \pi_j{\left(x \right)} &= \frac{\kappa_2 \pi_1{\left(x \right)} - \kappa_1 \pi_2{\left(x \right)}}{\kappa_1' \kappa_2 - \kappa_1 \kappa_2'} \\
    &= \frac{\kappa_2 \left(\kappa_1 w + \kappa_1' w' \right) - \kappa_1 \left(\kappa_2 w + \kappa_2' w' \right)}{\kappa_1' \kappa_2 - \kappa_1 \kappa_2'} \\
    &= w'.
\end{align*}

To see that $x$ is the only vector that satisfies these criteria, suppose that $\pi_i{\left(y \right)} = w$ and $\pi_j{\left(y \right)} = w'$ for some $y \in V \otimes W$. Since Equations \eqref{eqn:pi_1-relation} and \eqref{eqn:pi_2-relation} hold if each occurrence of $x$ is replaced with $y$, we have $\pi_1{\left(y \right)} = \kappa_1 w + \kappa_1' w'$ and $\pi_2{\left(y \right)} = \kappa_2 w + \kappa_2' w'$. Hence
\begin{displaymath}
    y = \left(e_1 + \mu_1 e_2 \right) \otimes \left(\kappa_1 w + \kappa_1' w' \right) + \left(e_1 + \mu_2 e_2 \right) \otimes \left(\kappa_2 w + \kappa_2' w' \right) = x.\qedhere
\end{displaymath}
\end{proof}

\begin{definition} \label{def:cliques}
For each $i \in I$ and $x \in V \otimes W$, we define the set
\begin{displaymath}
    \ell_i{\left(x \right)} = {\left\{\alpha \in V \otimes W \mid \pi_i{\left(\alpha \right)} = \pi_i{\left(x \right)}\right\}}.
\end{displaymath}
\end{definition}

\begin{lemma} \label{lm:clique-equivalence}
If $y \in \ell_i{\left(x \right)}$, then $\ell_i{\left(y \right)} = \ell_i{\left(x \right)}$.
\end{lemma}
\begin{proof}
This follows trivially from Definition \ref{def:cliques}.
\end{proof}

\begin{lemma} \label{lm:clique-intersection}
Let $x, y \in V \otimes W$ and let $i, j \in I$ such that $i \neq j$. Then $\ell_i{\left(x \right)}$ and $\ell_j{\left(y \right)}$ intersect at exactly one vector.
\end{lemma}
\begin{proof}
For any $\alpha \in V \otimes W$, we know that $\alpha \in \ell_i{\left(x \right)} \cap \ell_j{\left(y \right)}$ if and only if $\pi_i{\left(\alpha \right)} = \pi_i{\left(x \right)}$ and $\pi_j{\left(\alpha \right)} = \pi_j{\left(x \right)}$. This is true for precisely one vector in $V \otimes W$ according to Lemma~\ref{lm:two-projections-determines-tensor}.
\end{proof}

\begin{lemma} \label{lm:parallel-cliques}
Let $x, y \in V \otimes W$ and let $i \in I$. If $\ell_i{\left(x \right)} \cap \ell_i{\left(y \right)}$ is nonempty, then $y \in \ell_i{\left(x \right)}$.
\end{lemma}
\begin{proof}
Let $\alpha \in \ell_i{\left(x \right)} \cap \ell_i{\left(y \right)}$. Then $\pi_i{\left(\alpha \right)} = \pi_i{\left(x \right)} = \pi_i{\left(y \right)}$, so $y \in \ell_i{\left(x \right)}$ as required.
\end{proof}

\begin{definition} \label{def:hamming-union}
Let $\Delta$ be the subset of $V \otimes W$ defined by
\begin{displaymath}
    \Delta = {\left\{\left(e_1 + \mu_i e_2 \right) \otimes w \mid i \in I,\ w \in W \setminus \left\{0 \right\}\right\}},
\end{displaymath}
and let $\Gamma$ denote the Cayley digraph $\Cay{\left(V \otimes W,\ \Delta \right)}$.
\end{definition}

\begin{lemma} \label{lm:cliques}
Each $\ell_i{\left(x \right)}$ is a clique of size $q^m$ of the Cayley digraph $\Gamma$.
\end{lemma}
\begin{proof}
We will demonstrate this for $i = 1$; the proof is similar if $i \in \left\{2, 3, 4, 5, 6 \right\}$.

If $x \in V \otimes W$, then
\begin{align*}
    \ell_1{\left(x \right)} &= \left\{y \in V \otimes W \mid \pi_1{\left(y \right)} = \pi_1{\left(x \right)}\right\} \\
    &= {\left\{\left(e_1 + \mu_1 e_2 \right) \otimes \pi_1{\left(x \right)} + \left(e_1 + \mu_2 e_2 \right) \otimes w \mid w \in W \right\}},
\end{align*}
so $\left|\ell_1{\left(x \right)} \right| = \left|W \right| = q^m$.

Now let $y$ and $y'$ be distinct vectors in $\ell_1{\left(x \right)}$. Then $\pi_1{\left(y \right)} = \pi_1{\left(y' \right)} = \pi_1{\left(x \right)}$, but $\pi_2{\left(y \right)} \neq \pi_2{\left(y' \right)}$ since $y \neq y'$. Therefore
\begin{align*}
    y-y' &= \left(\left(e_1 + \mu_1 e_2 \right) \otimes \pi_1{\left(y \right)} + \left(e_1 + \mu_2 e_2 \right) \otimes \pi_2{\left(y \right)}\right) - \left(\left(e_1 + \mu_1 e_2 \right) \otimes \pi_1{\left(y' \right)} + \left(e_1 + \mu_2 e_2 \right) \otimes \pi_2{\left(y' \right)}\right) \\
    &= \left(e_1 + \mu_2 e_2 \right) \otimes \left(\pi_2{\left(y \right)} - \pi_2{\left(y' \right)}\right) \\
    &\in \Delta
\end{align*}
so $x$ and $y$ are adjacent in $\Gamma$.
\end{proof}

\begin{lemma} \label{lm:adjacency-implies-clique}
Let $x, y \in V \otimes W$ be adjacent in $\Gamma$. Then $\pi_i{\left(x \right)} = \pi_i{\left(y \right)}$ for some $i \in I$.
\end{lemma}
\begin{proof}
We have $y-x \in \Delta$, so $y-x = \left(e_1 + \mu_j e_2 \right) \otimes w$ for some $j \in I$ and $w \in W \setminus \left\{0 \right\}$. If $j=1$, then
\begin{align*}
    \left(e_1 + \mu_1 e_2 \right) \otimes w &= y-x \\
    &= \left(\left(e_1 + \mu_1 e_2 \right) \otimes \pi_1{\left(y \right)} + \left(e_1 + \mu_2 e_2 \right) \otimes \pi_2{\left(y \right)}\right) \\
    &\ \ \ \ \ \ \ - \left(\left(e_1 + \mu_1 e_2 \right) \otimes \pi_1{\left(x \right)} + \left(e_1 + \mu_2 e_2 \right) \otimes \pi_2{\left(x \right)}\right) \\
    &= \left(e_1 + \mu_1 e_2 \right) \otimes \left(\pi_1{\left(y \right)} - \pi_1{\left(x \right)}\right) + \left(e_1 + \mu_2 e_2 \right) \otimes \left(\pi_2{\left(y \right)} - \pi_2{\left(x \right)}\right)
\end{align*}
and therefore $\pi_2{\left(y \right)} - \pi_2{\left(x \right)} = 0$ as required. The proof is similar if $j \in \left\{2, 3, 4, 5, 6 \right\}$.
\end{proof}

\begin{lemma} \label{lm:no-other-cliques}
If $q > z$ and $\ell$ is a clique of size $q^m$ of $\Gamma$, then $\ell = \ell_i{\left(x \right)}$ for some $i \in I$ and some $x \in V \otimes W$.
\end{lemma}
\begin{proof}
Since $z^2 < q^2 \leqslant q^m$, there exist distinct vectors $x, y_1, y_2, \ldots, y_{z^2} \in \ell$. By Lemma~\ref{lm:adjacency-implies-clique}, each of the $y_k$ lies in at least one of the cliques $\left\{\ell_1{\left(x \right)}, \ell_2{\left(x \right)}, \ldots, \ell_z{\left(x \right)}\right\}$, so there must exist some $\ell_i{\left(x \right)}$ that contains at least $z$ elements of $\left\{y_1, y_2, \ldots, y_{z^2} \right\}$. Denote these by $x_1, x_2, \ldots, x_z$.

Now assume for contradiction that $\ell \neq \ell_i{\left(x \right)}$. Using Lemma~\ref{lm:cliques} we see that $\ell$ and $\ell_i{\left(x \right)}$ have the same cardinality, so $\ell$ is not contained in $\ell_i{\left(x \right)}$. Hence there must exist some $u \in \ell \setminus \ell_i{\left(x \right)}$. Since $u$ is adjacent to each of the vectors $x_1, \ldots, x_z$, by Lemma~\ref{lm:adjacency-implies-clique} there must exist some $i_1, i_2, \ldots, i_z \in I$ such that $x_k \in \ell_{i_k}{\left(u \right)}$ for all $k \in \left\{1, 2, \ldots, z \right\}$. Since $u \notin \ell_i{\left(x \right)}$, Lemma~\ref{lm:parallel-cliques} implies that $\ell_i{\left(x \right)}$ and $\ell_i{\left(u \right)}$ do not intersect. It follows that none of the $x_k$ lie in $\ell_i{\left(u \right)}$, so none of the $i_k$ are equal to $i$. But $\left|I \setminus \left\{i \right\}\right| = z-1$, so by the pigeonhole principle there exists some $j \in I \setminus \left\{i \right\}$ such that at least two of the $i_k$ are equal to $j$. Thus $\ell_j{\left(u \right)}$ contains at least two of the $x_k$, so $\ell_j{\left(u \right)} \cap \ell_i{\left(x \right)}$ contains at least two vectors. This contradicts Lemma~\ref{lm:clique-intersection}, so we conclude that $\ell = \ell_i{\left(x \right)}$.
\end{proof}

\begin{lemma} \label{lm:automorphisms-permute-parallel-classes}
Let $q > z$, let $x, y \in V \otimes W$, let $i \in I$ and let $\theta \in \Aut{\left(\Gamma \right)}$. Then there exist $x', y' \in V \otimes W$ and a unique $i' \in I$ such that $\left(\ell_i{\left(x \right)}\right)^\theta = \ell_{i'}{\left(x' \right)}$ and $\left(\ell_i{\left(y \right)}\right)^\theta = \ell_{i'}{\left(y' \right)}$.
\end{lemma}
\begin{proof}
Since $\ell_i{\left(x \right)}$ is a clique of size $q^m$ of $\Gamma$ by Lemma~\ref{lm:cliques}, it follows that $\left(\ell_i{\left(x \right)}\right)^\theta$ must also be a clique of size $q^m$ of $\Gamma$. Therefore, by Lemma~\ref{lm:no-other-cliques}, there exist some $x' \in V \otimes W$ and $i' \in I$ such that $\left(\ell_i{\left(x \right)}\right)^\theta = \ell_{i'}{\left(x' \right)}$, with $i'$ being unique as the existence of multiple possible values would contradict Lemma~\ref{lm:clique-intersection}. Similarly, there exist some $y' \in V \otimes W$ and $i'' \in I$ such that $\left(\ell_i{\left(y \right)}\right)^\theta = \ell_{i''}{\left(y' \right)}$.

First suppose $y \in \ell_i{\left(x \right)}$. Then we have $\ell_i{\left(y \right)} = \ell_i{\left(x \right)}$ from Lemma~\ref{lm:clique-equivalence} and hence $\left(\ell_i{\left(y \right)}\right)^\theta = \left(\ell_i{\left(x \right)}\right)^\theta = \ell_{i'}{\left(x' \right)}$, so the result holds in this case.

Now suppose instead that $y \notin \ell_i{\left(x \right)}$. By the contrapositive of Lemma~\ref{lm:parallel-cliques} we have $\ell_i{\left(x \right)} \cap \ell_i{\left(y \right)} = \varnothing$. Hence $\left(\ell_i{\left(x \right)}\right)^\theta \cap \left(\ell_i{\left(y \right)}\right)^\theta = \varnothing$. This means that $\ell_{i'}{\left(x' \right)} \cap \ell_{i''}{\left(y' \right)} = \varnothing$, which contradicts Lemma~\ref{lm:clique-intersection} unless $i' = i''$. The result follows.
\end{proof}

As a consequence of Lemma~\ref{lm:automorphisms-permute-parallel-classes}, provided $q > z$ we can define an action of $\Aut{\left(\Gamma \right)}$ on $I$ such that, if $\left(\ell_i{\left(x \right)}\right)^\theta = \ell_{i'}{\left(x' \right)}$ for any $x, x' \in V \otimes W$ and $\theta \in \Aut{\left(\Gamma \right)}$, then $i^\theta = i'$.

\begin{lemma} \label{lm:automorphisms-permute-cliques}
Let $q > z$, let $x, y \in V \otimes W$, and let $\theta \in \Aut{\left(\Gamma \right)}$. If $\pi_i{\left(x \right)} = \pi_i{\left(y \right)}$, then $\pi_{i^\theta}{\left(x^\theta \right)} = \pi_{i^\theta}{\left(y^\theta \right)}$.
\end{lemma}
\begin{proof}
Since $\ell_i{\left(x \right)}$ is a clique of size $q^m$ of $\Gamma$ by Lemma~\ref{lm:cliques}, it follows that $\left(\ell_i{\left(x \right)}\right)^\theta$ is also a clique of size $q^m$. Lemma~\ref{lm:no-other-cliques} therefore implies that $\left(\ell_i{\left(x \right)}\right)^\theta = \ell_{i'}{\left(x' \right)}$ for some $x' \in V \otimes W$ and $i' \in I$. Then $i^\theta = i'$, and since $x \in \ell_i{\left(x \right)}$ we know $x^\theta \in \left(\ell_i{\left(x \right)}\right)^\theta = \ell_{i^\theta}{\left(x' \right)}$. Lemma~\ref{lm:clique-equivalence} then asserts that $\ell_{i^\theta}{\left(x' \right)} = \ell_{i^\theta}{\left(x^\theta \right)}$. Now since $y \in \ell_i{\left(x \right)}$, we must have $y^\theta \in \left(\ell_i{\left(x \right)}\right)^\theta = \ell_{i^\theta}{\left(x^\theta \right)}$ and the result follows from Lemma~\ref{lm:clique-equivalence}.
\end{proof}

\begin{lemma} \label{lm:image-of-projection-to-zero}
Let $q > z$, let $x \in V \otimes W$ and let $\theta \in \Aut{\left(\Gamma \right)}_0$. If $\pi_i{\left(x \right)} = 0$ for some $i \in I$, then $\pi_{i^\theta}{\left(x^\theta \right)} = 0$.
\end{lemma}
\begin{proof}
Observing that $\pi_i{\left(0 \right)} = 0 = \pi_i{\left(x \right)}$, we have from Lemma~\ref{lm:automorphisms-permute-cliques} that $\pi_{i^\theta}{\left(x^\theta \right)} = \pi_{i^\theta}{\left(0^\theta \right)} = \pi_{i^\theta}{\left(0 \right)} = 0$.
\end{proof}

\begin{lemma} \label{lm:automorphism-linear-sufficiency}
Suppose $q > z$ and $q$ is prime, and let $\theta \in \Aut{\left(\Gamma \right)}$. If
\begin{displaymath}
    \pi_{i^\theta}{\left(\left(\alpha + \beta \right)^\theta \right)} = \pi_{i^\theta}{\left(\alpha^\theta \right)} + \pi_{i^\theta}{\left(\beta^\theta \right)}
\end{displaymath}
holds for all $\alpha, \beta \in V \otimes W$ and $i \in I$, then $\theta$ is a linear transformation of $V \otimes W$.
\end{lemma}
\begin{proof}
As $V \otimes W$ is a vector space over a field of prime order $q$, it suffices to show that $\theta$ preserves addition. Indeed, for all $\alpha, \beta \in V \otimes W$, we see that
\begin{align*}
    \left(\alpha + \beta \right)^\theta &= \left(e_1 + \mu_1 e_2 \right) \otimes \pi_1{\left(\left(\alpha + \beta \right)^\theta \right)} + \left(e_1 + \mu_2 e_2 \right) \otimes \pi_2{\left(\left(\alpha + \beta \right)^\theta \right)} \\
    &= \left(e_1 + \mu_1 e_2 \right) \otimes \left(\pi_1{\left(\alpha^\theta \right)} + \pi_1{\left(\beta^\theta \right)}\right) + \left(e_1 + \mu_2 e_2 \right) \otimes \left(\pi_2{\left(\alpha^\theta \right)} + \pi_2{\left(\beta^\theta \right)}\right) \\
    &= \left(e_1 + \mu_1 e_2 \right) \otimes \pi_1{\left(\alpha^\theta \right)} + \left(e_1 + \mu_2 e_2 \right) \otimes \pi_2{\left(\alpha^\theta \right)} \\
    &\ \ \ \ \ \ \ + \left(e_1 + \mu_1 e_2 \right) \otimes \pi_1{\left(\beta^\theta \right)} + \left(e_1 + \mu_2 e_2 \right) \otimes \pi_2{\left(\beta^\theta \right)} \\
    &= \alpha^\theta + \beta^\theta
\end{align*}
as required.
\end{proof}

\begin{lemma} \label{lm:automorphism-stabiliser-linear}
Suppose $q > z$ and $q$ is prime, and let $\theta \in \Aut{\left(\Gamma \right)}_0$. Then $\theta$ is a linear transformation of $V \otimes W$.
\end{lemma}
\begin{proof}
Let $x, y \in V \otimes W$, and let $i, j, k$ be distinct elements of $I$. In light of Lemma~\ref{lm:automorphism-linear-sufficiency}, it suffices to show that
\begin{displaymath}
    \pi_{i^\theta}{\left(\left(x + y \right)^\theta \right)} = \pi_{i^\theta}{\left(x^\theta \right)} + \pi_{i^\theta}{\left(y^\theta \right)}.
\end{displaymath}

According to Lemma~\ref{lm:projection-relations} there exist non-zero $\kappa_1, \kappa_2, \kappa'_1, \kappa'_2 \in \GF{\left(q \right)}$ such that
\begin{align}
    \pi_k{\left(u \right)} &= \kappa_1 \pi_i{\left(u \right)} + \kappa_2 \pi_j{\left(u \right)}, \label{eqn:specific-projection-relation} \\
    \pi_{k^\theta}{\left(u \right)} &= \kappa'_1 \pi_{i^\theta}{\left(u \right)} + \kappa'_2 \pi_{j^\theta}{\left(u \right)} \label{eqn:specific-projection-relation-image}
\end{align}
for all $u \in V \otimes W$.

According to Lemma~\ref{lm:two-projections-determines-tensor}, we can define $a \in V \otimes W$ such that $\pi_i{\left(a \right)} = \pi_i{\left(x \right)}$ and $\pi_j{\left(a \right)} = 0$. Hence Equation \eqref{eqn:specific-projection-relation} gives $\pi_k{\left(a \right)} = \kappa_1 \pi_i{\left(a \right)} + \kappa_2 \pi_j{\left(a \right)} = \kappa_1 \pi_i{\left(x \right)}$.

Now define $b \in V \otimes W$ such that $\pi_i{\left(b \right)} = \pi_i{\left(y \right)}$ and $\pi_j{\left(b \right)} = -\frac{\kappa_1}{\kappa_2} \pi_i{\left(y \right)}$. Hence Equation \eqref{eqn:specific-projection-relation} gives $\pi_k{\left(b \right)} = \kappa_1 \pi_i{\left(b \right)} + \kappa_2 \pi_j{\left(b \right)} = 0$.

Finally, define $c \in V \otimes W$ such that $\pi_i{\left(c \right)} = \pi_i{\left(x \right)} + \pi_i{\left(y \right)}$ and $\pi_j{\left(c \right)} = -\frac{\kappa_1}{\kappa_2} \pi_i{\left(y \right)}$. Observe that $\pi_j{\left(c \right)} = \pi_j{\left(b \right)}$, so we have
\begin{align}
    \pi_{j^\theta}{\left(c^\theta \right)} &= \pi_{j^\theta}{\left(b^\theta \right)} & \text{(using Lemma~\ref{lm:automorphisms-permute-cliques})} \notag \\
    &= -\frac{\kappa'_1}{\kappa'_2} \pi_{i^\theta}{\left(b^\theta \right)} + \frac{1}{\kappa'_2} \pi_{k^\theta}{\left(b^\theta \right)} & \text{(using Equation \eqref{eqn:specific-projection-relation-image})} \notag \\
    &= -\frac{\kappa'_1}{\kappa'_2} \pi_{i^\theta}{\left(b^\theta \right)}. & \text{(using Lemma~\ref{lm:image-of-projection-to-zero})} \label{eqn:pi_j-image}
\end{align}
Observe also that
\begin{displaymath}
    \pi_k{\left(c \right)} = \kappa_1 \pi_i{\left(c \right)} + \kappa_2 \pi_j{\left(c \right)} = \kappa_1 \pi_i{\left(x \right)} = \pi_k{\left(a \right)},
\end{displaymath}
which gives us
\begin{align}
    \pi_{k^\theta}{\left(c^\theta \right)} &= \pi_{k^\theta}{\left(a^\theta \right)} & \text{(using Lemma~\ref{lm:automorphisms-permute-cliques})} \notag \\
    &= \kappa'_1 \pi_{i^\theta}{\left(a^\theta \right)} + \kappa'_2 \pi_{j^\theta}{\left(a^\theta \right)} & \text{(using Equation \eqref{eqn:specific-projection-relation-image})} \notag \\
    &= \kappa'_1 \pi_{i^\theta}{\left(a^\theta \right)}. & \text{(using Lemma~\ref{lm:image-of-projection-to-zero})} \label{eqn:pi_k-image}
\end{align}
Making use of Lemma~\ref{lm:projections-preserve-addition-and-multiplication}, we have
\begin{displaymath}
    \pi_i{\left(x + y \right)} = \pi_i{\left(x \right)} + \pi_i{\left(y \right)} = \pi_i{\left(c \right)}.
\end{displaymath}
We therefore conclude that
\begin{align*}
    \pi_{i^\theta}{\left(\left(x + y \right)^\theta \right)}
    &= \pi_{i^\theta}{\left(c^\theta \right)} & \text{(using Lemma~\ref{lm:automorphisms-permute-cliques})} \\
    &= \frac{1}{\kappa'_1} \pi_{k^\theta}{\left(c^\theta \right)} -\frac{\kappa'_2}{\kappa'_1} \pi_{j^\theta}{\left(c^\theta \right)} & \text{(using Equation \eqref{eqn:specific-projection-relation-image})} \\
    &= \pi_{i^\theta}{\left(a^\theta \right)} + \pi_{i^\theta}{\left(b^\theta \right)} & \text{(using Equations \eqref{eqn:pi_j-image} and \eqref{eqn:pi_k-image})} \\
    &= \pi_{i^\theta}{\left(x^\theta \right)} + \pi_{i^\theta}{\left(y^\theta \right)} & \text{(using Lemma~\ref{lm:automorphisms-permute-cliques} on each term)}
\end{align*}
as required.
\end{proof}

\begin{lemma} \label{lm:automorphism-stabiliser-fixes-tensor-decomposition}
Suppose $q > z$ and $q$ is prime, and let $\theta \in \Aut{\left(\Gamma \right)}_0$. Then $\theta \in \GL{\left(2, q \right)} \circ \GL{\left(m, q \right)}$.
\end{lemma}
\begin{proof}
We first note that $\theta$ is a linear transformation of $V \otimes W$ by Lemma~\ref{lm:automorphism-stabiliser-linear}, so it preserves the set $\Delta$ by Lemma~\ref{lm:endomorphism-automorphism}. If $\theta$ preserves the set of all simple tensors then by Lemma~\ref{lm:preserves-simple-tensors-preserves-decomp} it would preserve the tensor product structure of $V \otimes W$, and would hence lie in $\GL{\left(2, q \right)} \circ \GL{\left(m, q \right)}$ due to the asymmetry of $\Delta$ on this tensor product structure. Therefore, letting $x$ be an arbitrary simple tensor in $V \otimes W$, it suffices to show that $x^\theta$ is a simple tensor.

Since $x$ is a simple tensor, we must have $x = \left(\nu_1 e_1 + \nu_2 e_2 \right) \otimes w$ for some $\nu_1, \nu_2 \in \GF{\left(q \right)}$ and $w \in W$. Suppose first that $-\mu_2 \nu_1 + \nu_2 = 0$. Then
\begin{displaymath}
    x = \left(\nu_1 e_1 + \mu_2 \nu_1 e_2 \right) \otimes w = \left(e_1 + \mu_2 e_2 \right) \otimes \nu_1 w \in \Delta \cup {\left\{0 \right\}}.
\end{displaymath}
Since $\Delta \cup \left\{0 \right\}$ is preserved by $\theta$ and contains only simple tensors, it follows that $x^\theta$ is a simple tensor as required.

Now suppose instead that $-\mu_2 \nu_1 + \nu_2 \neq 0$. It is easy to check using the definition of $x$ that
\begin{displaymath}
    x = \left(e_1 + \mu_1 e_2 \right) \otimes \left(\frac{-\mu_2 \nu_1 + \nu_2}{\mu_1 - \mu_2} w \right) + \left(e_1 + \mu_2 e_2 \right) \otimes {\left(\frac{\mu_1 \nu_1 - \nu_2}{\mu_1 - \mu_2} w \right)}.
\end{displaymath}
It follows immediately that
\begin{align}
    \pi_2{\left(x \right)} = \frac{\mu_1 \nu_1 - \nu_2}{-\mu_2 \nu_1 + \nu_2} \pi_1{\left(x \right)}. \label{eqn:projection-ratio}
\end{align}

By Lemma~\ref{lm:two-projections-determines-tensor} we can define $y \in V \otimes W$ such that $\pi_1{\left(y \right)} = \pi_1{\left(x \right)}$ and $\pi_3{\left(y \right)} = 0$. Then Lemmas \ref{lm:automorphisms-permute-cliques} and \ref{lm:image-of-projection-to-zero} respectively give
\begin{align}
    \pi_{1^\theta}{\left(y^\theta \right)} &= \pi_{1^\theta}{\left(x^\theta \right)}, \label{eqn:matching-projection-1-image} \\
    \pi_{3^\theta}{\left(y^\theta \right)} &= 0. \label{eqn:projection-3-image-zero}
\end{align}

According to Lemma~\ref{lm:projection-relations} there exist $\kappa_1, \kappa_3, \kappa'_1, \kappa'_3, \iota_1, \iota_2, \iota'_1, \iota'_2 \in \GF{\left(q \right)}$ such that
\begin{align}
    \pi_2{\left(y \right)} &= \kappa_1 \pi_1{\left(y \right)} + \kappa_3 \pi_3{\left(y \right)}, \label{eqn:specific-projection-relation-simple-tensors} \\
    \pi_{2^\theta}{\left(y^\theta \right)} &= \kappa'_1 \pi_{1^\theta}{\left(y^\theta \right)} + \kappa'_3 \pi_{3^\theta}{\left(y^\theta \right)}, \label{eqn:specific-projection-relation-image-simple-tensors} \\
    \pi_1{\left(x^\theta \right)} &= \iota_1 \pi_{1^\theta}{\left(x^\theta \right)} + \iota_2 \pi_{2^\theta}{\left(x^\theta \right)}, \label{eqn:pi-1-of-image} \\
    \pi_2{\left(x^\theta \right)} &= \iota'_1 \pi_{1^\theta}{\left(x^\theta \right)} + \iota'_2 \pi_{2^\theta}{\left(x^\theta \right)}, \label{eqn:pi-2-of-image}
\end{align}
where $\kappa_1 \neq 0$. Hence we have
\begin{align*}
    \pi_2{\left(\frac{\mu_1 \nu_1 - \nu_2}{\kappa_1 \left(-\mu_2 \nu_1 + \nu_2 \right)} y \right)} &= \frac{\mu_1 \nu_1 - \nu_2}{\kappa_1 \left(-\mu_2 \nu_1 + \nu_2 \right)} \pi_2{\left(y \right)} &\text{(using Lemma~\ref{lm:projections-preserve-addition-and-multiplication})} \\
    &= \frac{\mu_1 \nu_1 - \nu_2}{\kappa_1 \left(-\mu_2 \nu_1 + \nu_2 \right)} \left(\kappa_1 \pi_1{\left(y \right)} + \kappa_3 \pi_3{\left(y \right)} \right) &\text{(using Equation \eqref{eqn:specific-projection-relation-simple-tensors})} \\
    &= \frac{\mu_1 \nu_1 - \nu_2}{-\mu_2 \nu_1 + \nu_2} \pi_1 {\left(x \right)} &\text{(since $\pi_1{\left(y \right)} = \pi_1{\left(x \right)}$ and $\pi_3{\left(y \right)} = 0$)} \\
    &= \pi_2{\left(x \right)}. &\text{(using Equation \eqref{eqn:projection-ratio})}
\end{align*}

Therefore
\begin{align}
    \pi_{2^\theta}{\left(x^\theta \right)} &= \pi_{2^\theta}{\left(\left(\frac{\mu_1 \nu_1 - \nu_2}{\kappa_1 \left(-\mu_2 \nu_1 - \nu_2 \right)} y \right)^\theta \right)} &\text{(using Lemma~\ref{lm:automorphisms-permute-cliques})} \notag \\
    &= \pi_{2^\theta}{\left(\frac{\mu_1 \nu_1 - \nu_2}{\kappa_1 \left(-\mu_2 \nu_1 - \nu_2 \right)} y^\theta \right)} &\text{(using Lemma~\ref{lm:automorphism-stabiliser-linear})} \notag \\
    &= \frac{\mu_1 \nu_1 - \nu_2}{\kappa_1 \left(-\mu_2 \nu_1 - \nu_2 \right)} \pi_{2^\theta}{\left(y^\theta \right)} &\text{(using Lemma~\ref{lm:projections-preserve-addition-and-multiplication})} \notag \\
    &= \frac{\mu_1 \nu_1 - \nu_2}{\kappa_1 \left(-\mu_2 \nu_1 - \nu_2 \right)} \left(\kappa'_1 \pi_{1^\theta}{\left(y^\theta \right)} + \kappa'_3 \pi_{3^\theta}{\left(y^\theta \right)}\right) &\text{(using Equation \eqref{eqn:specific-projection-relation-image-simple-tensors})} \notag \\
    &= \frac{\mu_1 \nu_1 - \nu_2}{\kappa_1 \left(-\mu_2 \nu_1 - \nu_2 \right)} \kappa'_1 \pi_{1^\theta}{\left(x^\theta \right)} &\text{(using Equations \eqref{eqn:matching-projection-1-image} and \eqref{eqn:projection-3-image-zero})} \label{eqn:pi-2-image}
\end{align}
and finally
\begin{align*}
    x^\theta &= \left(e_1 + \mu_1 e_2 \right) \otimes \pi_1{\left(x^\theta \right)} + \left(e_1 + \mu_2 e_2 \right) \otimes \pi_2{\left(x^\theta \right)} \\
    &= \left(e_1 + \mu_1 e_2 \right) \otimes \left(\iota_1 \pi_{1^\theta}{\left(x^\theta \right)} + \iota_2 \pi_{2^\theta}{\left(x^\theta \right)}\right) \\
    &\ \ \ \ \ \ + \left(e_1 + \mu_2 e_2 \right) \otimes \left(\iota'_1 \pi_{1^\theta}{\left(x^\theta \right)} + \iota'_2 \pi_{2^\theta}{\left(x^\theta \right)}\right)  &\text{(using Equations \eqref{eqn:pi-1-of-image} and \eqref{eqn:pi-2-of-image})} \\
    &= \left(e_1 + \mu_1 e_2 \right) \otimes \left(\iota_1 \pi_{1^\theta}{\left(x^\theta \right)} + \iota_2 \frac{\mu_1 \nu_1 - \nu_2}{\kappa_1 \left(-\mu_2 \nu_1 - \nu_2 \right)} \kappa'_1 \pi_{1^\theta}{\left(x^\theta \right)}\right) \\
    &\ \ \ \ \ \ + \left(e_1 + \mu_2 e_2 \right) \otimes \left(\iota'_1 \pi_{1^\theta}{\left(x^\theta \right)} + \iota'_2 \frac{\mu_1 \nu_1 - \nu_2}{\kappa_1 \left(-\mu_2 \nu_1 - \nu_2 \right)} \kappa'_1 \pi_{1^\theta}{\left(x^\theta \right)}\right) &\text{(using Equation \eqref{eqn:pi-2-image})} \\
    &= \left(e_1 + \mu_1 e_2 \right) \otimes \left(\iota_1 + \iota_2 \frac{\mu_1 \nu_1 - \nu_2}{\kappa_1 \left(-\mu_2 \nu_1 - \nu_2 \right)} \kappa'_1 \right) \pi_{1^\theta}{\left(x^\theta \right)} \\
    &\ \ \ \ \ \ + \left(e_1 + \mu_2 e_2 \right) \otimes \left(\iota'_1 + \iota'_2 \frac{\mu_1 \nu_1 - \nu_2}{\kappa_1 \left(-\mu_2 \nu_1 - \nu_2 \right)} \kappa'_1 \right) \pi_{1^\theta}{\left(x^\theta \right)},
\end{align*}
so $x^\theta$ is equal to
\begin{displaymath}
    \left(\left(\iota_1 + \iota_2 \frac{\mu_1 \nu_1 - \nu_2}{\kappa_1 \left(-\mu_2 \nu_1 - \nu_2 \right)} \kappa'_1 \right) \left(e_1 + \mu_1 e_2 \right) + \left(\iota'_1 + \iota'_2 \frac{\mu_1 \nu_1 - \nu_2}{\kappa_1 \left(-\mu_2 \nu_1 - \nu_2 \right)} \kappa'_1 \right) \left(e_1 + \mu_2 e_2 \right)\right) \otimes \pi_{1^\theta}{\left(x^\theta \right)}
\end{displaymath}
and is therefore a simple tensor as required.
\end{proof}

\begin{proposition} \label{prop:automorphism-group-affine-decomp}
Suppose $q > z$ and $q$ is prime, let $\mu_1$, $\mu_2$, $\mu_3$, $\mu_4$ be distinct elements of $\GF{\left(q \right)}$, and let $\Gamma$ be as defined in Definition \ref{def:hamming-union}. Then $\Aut{\left(\Gamma \right)} \leqslant \left(V \otimes W \right) \rtimes \left(\GL{\left(2, q \right)} \circ \GL{\left(m, q \right)}\right)$.
\end{proposition}
\begin{proof}
This follows from Lemma~\ref{lm:automorphism-stabiliser-fixes-tensor-decomposition} and the fact that the group of translations of $V \otimes W$ is transitive.
\end{proof}

\section{Specific cases of the permutation group $G{\left(m, q \right)}$} \label{sc:specific-values}

We examine first the specific case where $q = 5$, so let $G = G{\left(m, 5 \right)}$ and let $F = \GF{\left(5 \right)}$.
From Lemma~\ref{lm:suborbit-equivalence} we have $\Delta_1 = \Delta_4$ and $\Delta_2 = \Delta_3$, so $G$ has orbitals corresponding to the suborbits $\delta$, $\Delta_A$, $\Delta_B$, $\Delta_1$ and $\Delta_2$. Thus $G$ has rank $5$.

\begin{theorem} \label{thm:q=5}
The group $G{\left(m, 5 \right)}$ is not the automorphism group of any digraph for any $m \geqslant 2$.
\end{theorem}
\begin{proof}
Let $G = G{\left(m, 5 \right)}$. It suffices to show that $G$ is not the automorphism group of any union of its four non-trivial orbital digraphs. We know already from Lemma~\ref{lm:Gamma_A,B} that $G$ is not the automorphism group of $\Gamma_A$ or $\Gamma_B$. Since $2^2 = 4 = -1$ in the field $F$, we have from Lemma~\ref{lm:Gamma_1,i} that $G$ is not the automorphism group of $\Gamma_1$ or $\Gamma_2$ either.

Consider the following orbital digraph union:
\begin{align*}
\Delta_A \cup \Delta_1 &= \left\{e_1 \otimes w,\ e_2 \otimes w,\ \left(e_1 \pm e_2 \right) \otimes w \mid w \in W \setminus \left\{0 \right\}\right\} \\
&= \left\{\left(1, 0 \right) \otimes w,\ \left(0, 1 \right) \otimes w,\ \left(1, 1 \right) \otimes w,\ \left(1, -1 \right) \otimes w \mid w \in W \setminus \left\{0 \right\}\right\}
\end{align*}
with respect to the basis $\left\{e_1, e_2 \right\}$ of $V$. Consider the image of this union under the linear transformation $\theta = \begin{bsmallmatrix} 1 & 1 \\ 1 & -1 \end{bsmallmatrix} \circ I$, where $I$ denotes the $m \times m$ identity matrix. If we take, for example, a vector of the form $x = \left(1, 1 \right) \otimes w$ for some $w \in W \setminus \left\{0 \right\}$, then $x \in \Delta_A \cup \Delta_1$ and its image under $\theta$ is given by
\begin{displaymath}
x^\theta = \left(\left(1, 1 \right) \begin{bmatrix} 1 & 1 \\ 1 & -1 \end{bmatrix} \right) \otimes \left(wI \right) = \left(2, 0 \right) \otimes w = \left(1, 0 \right) \otimes \left(2w \right) \in \Delta_A \cup \Delta_1.
\end{displaymath}
It is easy to check using a similar calculation that if $x$ is a vector of the form $\left(1, 0 \right) \otimes w$ or $\left(0, 1 \right) \otimes w$ or $\left(1, -1 \right) \otimes w$, then the image of $x$ under $\theta$ lies in $\Delta_A \cup \Delta_1$.

We have now seen that $\theta$ preserves the set $\Delta_A \cup \Delta_1$, which means $\theta$ is an automorphism of $\Gamma_A \cup \Gamma_1$ by Lemma~\ref{lm:endomorphism-automorphism}. Since $\theta$ fixes the zero vector but does not lie in $G_0 = D_8 \circ \GL{\left(m, 5 \right)}$, it cannot be an element of $G$. Therefore $G$ is not the full automorphism group of $\Gamma_A \cup \Gamma_1$.

Now observe that
\begin{align*}
\Delta_A \cup \Delta_2 &= \left\{e_1 \otimes w,\ e_2 \otimes w,\ \left(e_1 \pm 2 e_2 \right) \otimes w \mid w \in W \setminus \left\{0 \right\}\right\} \\
&= {\left\{\left(1, 0 \right) \otimes w,\ \left(0, 1 \right) \otimes w,\ \left(1, 2 \right) \otimes w,\ \left(1, -2 \right) \otimes w \mid w \in W \setminus \left\{0 \right\}\right\}}.
\end{align*}
It can be checked that the linear transformation $\begin{bsmallmatrix} 1 & 2 \\ 2 & 1 \end{bsmallmatrix} \circ I$ preserves the set $\Delta_A \cup \Delta_2$, hence it is an automorphism of $\Gamma_A \cup \Gamma_2$ by Lemma~\ref{lm:endomorphism-automorphism}. Once again, this is an automorphism that does not lie in $G$, so $G$ is not the full automorphism group of $\Gamma_A \cup \Gamma_2$.

Finally, we have
\begin{align*}
\Delta_1 \cup \Delta_2 &= \left\{\left(e_1 \pm e_2 \right) \otimes w,\ \left(e_1 \pm 2 e_2 \right) \otimes w \mid w \in W \setminus \left\{0 \right\}\right\} \\
&= {\left\{\left(1, 1 \right) \otimes w,\ \left(1, -1 \right) \otimes w,\ \left(1, 2 \right) \otimes w,\ \left(1, -2 \right) \otimes w \mid w \in W \setminus \left\{0 \right\}\right\}}.
\end{align*}
It can be checked that the linear transformation $\begin{bsmallmatrix} 1 & 0 \\ 0 & 2 \end{bsmallmatrix} \circ I$ fixes the set $\Delta_1 \cup \Delta_2$, hence it is an automorphism of $\Gamma_1 \cup \Gamma_2$ by Lemma~\ref{lm:endomorphism-automorphism}. This automorphism does not lie in $G$, so $G$ is not the full automorphism group of $\Gamma_1 \cup \Gamma_2$.

The remaining orbital digraph unions are $\Delta_A \cup \Delta_B$, $\Delta_B \cup \Delta_1$ and $\Delta_B \cup \Delta_2$, as well as each union of three non-trivial orbital digraphs. However, these are the complements of orbital digraph unions that we have already checked, and every digraph has the same automorphism group as its complement. Hence we conclude that the group $G{\left(m, 5 \right)}$ is not the automorphism group of any digraph, as it is not the automorphism group of any union of its non-trivial orbital digraphs.
\end{proof}

Next, take $q=7$, let $m \geqslant 2$ be an integer, and let $G = G{\left(m, 7 \right)}$. As we are now working in the field $F = \GF {\left(7 \right)}$, Lemma~\ref{lm:suborbit-equivalence} gives us $\Delta_2 = \Delta_3 = \Delta_4 = \Delta_5$ and $\Delta_1 = \Delta_6$. Therefore the orbitals of $G$ are those corresponding to the suborbits $\delta$, $\Delta_A$, $\Delta_B$, $\Delta_1$ and $\Delta_2$, so $G$ once again has rank $5$.

\begin{theorem} \label{thm:q=7}
The group $G{\left(m, 7 \right)}$ is not the automorphism group of any digraph for any $m \geqslant 2$.
\end{theorem}
\begin{proof}
The orbital digraphs $\Gamma_A$, $\Gamma_B$ and $\Gamma_1$ do not have $G$ as their automorphism group according to Lemmas \ref{lm:Gamma_A,B} and \ref{lm:Gamma_1,i}.

The suborbit $\Delta_2$ as defined in Section \ref{sc:suborbits} is given by
\begin{align*}
\Delta_2 &= \left\{\left(e_1 \pm 2 e_2 \right) \otimes w,\ \left(e_1 \pm 2^{-1} e_2 \right) \otimes w \mid w \in W \setminus \left\{0 \right\}\right\} \\
&= \left\{\left(1, 2 \right) \otimes w,\ \left(1, 3 \right) \otimes w,\ \left(1, 4 \right) \otimes w,\ \left(1, 5 \right) \otimes w \mid w \in W \setminus \left\{0 \right\}\right\}.
\end{align*}
It can be checked that this set is preserved by the linear transformation $\begin{bsmallmatrix} 1 & 2 \\ 2 & 1 \end{bsmallmatrix} \circ I$. Therefore by Lemma~\ref{lm:endomorphism-automorphism} this transformation is an automorphism of $\Gamma_2$, and since it does not lie in $G$ we can conclude that $\Aut(\Gamma_2) \neq G$.

Similarly, observe that
\begin{displaymath}
\begin{array}{l@{}l@{}l}
\Delta_A \cup \Delta_1 &\ = \big\{&e_1 \otimes w,\ e_2 \otimes w,\ \left(e_1 \pm e_2 \right) \otimes w \mid w \in W \setminus {\left\{0 \right\}} \big\} \\
&\ = \big\{&\left(1, 0 \right) \otimes w,\ \left(0, 1 \right) \otimes w,\ \left(1, 1 \right) \otimes w,\ \left(1, 6 \right) \otimes w \mid w \in W \setminus {\left\{0 \right\}} \big\}, \\
\Delta_A \cup \Delta_2 &\ = \big\{&e_1 \otimes w,\ e_2 \otimes w,\ \left(e_1 \pm 2 e_2 \right) \otimes w,\ \left(e_1 \pm 2^{-1} e_2 \right) \otimes w \mid w \in W \setminus {\left\{0 \right\}} \big\} \\
&\ = \big\{&\left(1, 0 \right) \otimes w,\ \left(0, 1 \right) \otimes w,\ \left(1, 2 \right) \otimes w,\ \left(1, 3 \right) \otimes w,\ \left(1, 4 \right) \otimes w, \\
&&\left(1, 5 \right) \otimes w \mid w \in W \setminus {\left\{0 \right\}} \big\}, \\
\Delta_1 \cup \Delta_2 &\ = \big\{&\left(e_1 \pm e_2 \right) \otimes w,\ \left(e_1 \pm 2 e_2 \right) \otimes w,\ \left(e_1 \pm 2^{-1} e_2 \right) \otimes w \mid w \in W \setminus {\left\{0 \right\}} \big\} \\
&\ = \big\{&\left(1, \mu \right) \otimes w \mid \mu \in F \setminus {\left\{0 \right\}},\ w \in W \setminus {\left\{0 \right\}} \big\}.
\end{array}
\end{displaymath}
It can be checked that these sets are fixed by the linear transformations $\begin{bsmallmatrix} 1 & 1 \\ 1 & -1 \end{bsmallmatrix} \circ I$, $\begin{bsmallmatrix} 1 & 2 \\ 2 & 1 \end{bsmallmatrix} \circ I$ and $\begin{bsmallmatrix} 1 & 0 \\ 0 & 2 \end{bsmallmatrix} \circ I$, respectively. As these transformations do not lie in $G$, it follows from Lemma~\ref{lm:endomorphism-automorphism} that $G$ is not the automorphism group of $\Gamma_A \cup \Gamma_1$, nor $\Gamma_A \cup \Gamma_2$, nor $\Gamma_1 \cup \Gamma_2$.

Each of the remaining unions of non-trivial orbital digraphs is the complement of one of the unions we have already checked. Hence $G{\left(m, 7 \right)}$ is not the automorphism group of any digraph since it is not the full automorphism group of any union of its non-trivial orbital digraphs.
\end{proof}

We now consider the case where $q = 13$, so let $G = G{\left(m, 13 \right)}$ and observe from Lemma~\ref{lm:suborbit-equivalence} that
\begin{align*}
&\Delta_1 = \Delta_{12}, \\
&\Delta_2 = \Delta_6 = \Delta_7 = \Delta_{11}, \\
&\Delta_3 = \Delta_4 = \Delta_9 = \Delta_{10}, \\
&\Delta_5 = \Delta_8.
\end{align*}
Therefore the set
\begin{displaymath}
\mathcal{O} = \left\{\delta,\ \Delta_A,\ \Delta_B,\ \Delta_1,\ \Delta_2,\ \Delta_3,\ \Delta_5 \right\}
\end{displaymath}
is a complete set of distinct suborbits of $G$, so the rank of $G$ is $7$. The six non-trivial orbital digraphs are $\Gamma_A$, $\Gamma_B$, $\Gamma_1$, $\Gamma_2$, $\Gamma_3$ and $\Gamma_5$.

\begin{theorem} \label{thm:q=13}
The group $G{\left(m, 13 \right)}$ is not the automorphism group of any digraph for any $m \geqslant 2$.
\end{theorem}
\begin{proof}
First we examine the non-trivial orbital digraph unions that do not include $\Gamma_A$ or $\Gamma_B$. It can be checked that for each union of the orbital digraphs $\Gamma_1$, $\Gamma_2$, $\Gamma_3$ and $\Gamma_5$, the invertible linear transformation given in Table~\ref{tbl:q=13,endomorphisms} preserves the corresponding union of suborbits but does not lie in $G$. Each is an automorphism of the respective orbital digraph union by Lemma~\ref{lm:endomorphism-automorphism}.

\begin{table}
\centering
\begin{tabular}{cc||cc}
\hline
Orbital digraph union & Automorphism & Orbital digraph union & Automorphism \\
\hline
$\Gamma_1$ & $\begin{bmatrix} 1 & 2 \\ 2 & 1 \end{bmatrix} \circ I$ & $\Gamma_2 \cup \Gamma_5$ & $\begin{bmatrix} 1 & 0 \\ 0 & 4 \end{bmatrix} \circ I$ \\
$\Gamma_2$ & $\begin{bmatrix} 1 & 1 \\ 5 & -5 \end{bmatrix} \circ I$ & $\Gamma_3 \cup \Gamma_5$ & $\begin{bmatrix} 1 & 2 \\ 2 & 1 \end{bmatrix} \circ I$ \\
$\Gamma_3$ & $\begin{bmatrix} 1 & 1 \\ 5 & -5 \end{bmatrix} \circ I$ & $\Gamma_1 \cup \Gamma_2 \cup \Gamma_3$ & $\begin{bmatrix} 1 & 0 \\ 0 & 2 \end{bmatrix} \circ I$ \\
$\Gamma_5$ & $\begin{bmatrix} 1 & 1 \\ 1 & -1 \end{bmatrix} \circ I$ & $\Gamma_1 \cup \Gamma_2 \cup \Gamma_5$ & $\begin{bmatrix} 1 & 4 \\ 4 & -1 \end{bmatrix} \circ I$ \\
$\Gamma_1 \cup \Gamma_2$ & $\begin{bmatrix} 1 & 4 \\ 4 & -1 \end{bmatrix} \circ I$ & $\Gamma_1 \cup \Gamma_3 \cup \Gamma_5$ & $\begin{bmatrix} 1 & 2 \\ 2 & 1 \end{bmatrix} \circ I$ \\
$\Gamma_1 \cup \Gamma_3$ & $\begin{bmatrix} 1 & 0 \\ 0 & 4 \end{bmatrix} \circ I$ & $\Gamma_2 \cup \Gamma_3 \cup \Gamma_5$ & $\begin{bmatrix} 1 & 1 \\ 1 & -1 \end{bmatrix} \circ I$ \\
$\Gamma_1 \cup \Gamma_5$ & $\begin{bmatrix} 1 & 0 \\ 0 & 5 \end{bmatrix} \circ I$ & $\Gamma_1 \cup \Gamma_2 \cup \Gamma_3 \cup \Gamma_5$ & $\begin{bmatrix} 1 & 0 \\ 0 & 2 \end{bmatrix} \circ I$ \\
$\Gamma_2 \cup \Gamma_3$ & $\begin{bmatrix} 1 & 1 \\ 5 & -5 \end{bmatrix} \circ I$ \\
\hline
\end{tabular}
\caption{A linear automorphism not in $G{\left(m, 13 \right)}$ for each of the given unions of orbital digraphs. The $m \times m$ identity matrix is denoted by $I$.}
\label{tbl:q=13,endomorphisms}
\end{table}

Suppose now that $\Gamma = \Gamma_B \cup \Gamma'$, where $\Gamma'$ is some union of $\Gamma_1$, $\Gamma_2$, $\Gamma_3$ and $\Gamma_5$. We have just seen that there exists some $\theta \in \GL{\left(2, q \right)} \circ \GL{\left(m, q \right)}$ that is an automorphism of $\Gamma'$ but that does not lie in $G$. But since $\theta$ preserves the set $\Delta_B$ of all non-simple tensors, it must also be an automorphism of $\Gamma_B$ by Lemma~\ref{lm:endomorphism-automorphism}. It follows that $\theta \in \Aut{\left(\Gamma \right)}$.

Finally, if a union of non-trivial orbital digraphs of $G$ includes $\Gamma_A$, then its complement does not include $\Gamma_A$ and has therefore already been checked.

Since $G$ is not the full automorphism group of any union of its non-trivial orbital digraphs, we conclude that $G$ is not the automorphism group of any digraph.
\end{proof}

\begin{theorem} \label{thm:2-closed}
The groups $G{\left(m, 5 \right)}$, $G{\left(m, 7 \right)}$ and $G{\left(m, 13 \right)}$ are $2$-closed for all $m \geqslant 2$.
\end{theorem}
\begin{proof}
First let $G = G{\left(m, 5 \right)}$, let $\Delta = \Delta_1 \cup \Delta_2$, and let $\Gamma = \Gamma_1 \cup \Gamma_2 = \Cay{\left(V \otimes W,\ \Delta \right)}$. If we define $\mu_1 = 1$, $\mu_2 = 2$, $\mu_3 = 3$ and $\mu_4 = 4$, then we have
\begin{displaymath}
    \Delta = {\left\{\left(e_1 + \mu_i e_2 \right) \otimes w \mid i \in {\left\{1, 2, 3, 4 \right\}},\ w \in W \setminus \left\{0 \right\}\right\}},
\end{displaymath}
so Proposition~\ref{prop:automorphism-group-affine-decomp} asserts that $\Aut{\left(\Gamma \right)} \leqslant \left(V \otimes W \right) \rtimes \left(\GL{\left(2, 5 \right)} \circ \GL{\left(m, 5 \right)}\right)$. Therefore
\begin{displaymath}
    G^{\left(2 \right)} \leqslant \Aut{\left(\Gamma_1 \right)} \cap \Aut{\left(\Gamma_2 \right)} \leqslant \Aut{\left(\Gamma \right)} \leqslant \left(V \otimes W \right) \rtimes {\left(\GL{\left(2, 5 \right)} \circ \GL{\left(m, 5 \right)}\right)}.
\end{displaymath}

Let $A \circ B \in G^{\left(2 \right)}_0$. It suffices to show that $A \circ B \in G_0$. Since $A \circ B \in G^{\left(2 \right)} \leqslant \Aut{\left(\Gamma_1 \right)}$, we know from Lemma~\ref{lm:endomorphism-automorphism} that $A \circ B$ preserves the suborbit $\Delta_1$. It is easy then to see that $A$ must preserve the set
\begin{displaymath}
    V_1 = \left(\left<e_1 + e_2 \right> \cup \left<e_1 + 4 e_2 \right> \right) \setminus {\left\{0 \right\}}.
\end{displaymath}
Similarly, $A \circ B$ preserves the suborbit $\Delta_2$, so $A$ must preserve the set
\begin{displaymath}
    V_2 = \left(\left<e_1 + 2 e_2 \right> \cup \left<e_1 + 3 e_2 \right> \right) \setminus {\left\{0 \right\}}.
\end{displaymath}
It can be checked computationally that the stabilisers of $V_1$ and $V_2$ in $\GL{\left(2, 5 \right)}$ intersect only at $D_8$, so $A \circ B$ lies in $G_0$ as required.

Now let $G = G{\left(m, 7 \right)}$. If we define $\mu_1 = 2$, $\mu_2 = 3$, $\mu_3 = 4$ and $\mu_4 = 5$, then
\begin{displaymath}
    \Delta_2 = {\left\{\left(e_1 + \mu_i e_2 \right) \otimes w \mid i \in {\left\{1, 2, 3, 4 \right\}},\ w \in W \setminus \left\{0 \right\}\right\}},
\end{displaymath}
and hence
\begin{displaymath}
    G^{\left(2 \right)} \leqslant \Aut{\left(\Gamma_2 \right)} \leqslant \left(V \otimes W \right) \rtimes {\left(\GL{\left(2, 7 \right)} \circ \GL{\left(m, 7 \right)}\right)}
\end{displaymath}
by the same reasoning as above. Letting $A \circ B \in G^{\left(2 \right)}_0$, it suffices to show that $A \circ B \in G_0$. By reasoning similar to that above, $A$ must preserve the sets
\begin{displaymath}
    V_1 = \left(\left<e_1 \right> \cup \left<e_2 \right> \right) \setminus \left\{0 \right\}
\end{displaymath}
and
\begin{displaymath}
    V_A = \left(\left<e_1 + e_2 \right> \cup \left<e_1 + 4 e_2 \right> \right) \setminus {\left\{0 \right\}}.
\end{displaymath}
It can be checked computationally that the stabilisers of $V_A$ and $V_1$ in $\GL{\left(2, 7 \right)}$ intersect only at $D_8$, so $A \circ B$ lies in $G_0$ as required.

Finally, if $G = G{\left(m, 13 \right)}$ then define $\mu_1 = 2$, $\mu_2 = 6$, $\mu_3 = 7$ and $\mu_4 = 11$. Then
\begin{displaymath}
    \Delta_2 = {\left\{\left(e_1 + \mu_i e_2 \right) \otimes w \mid i \in {\left\{1, 2, 3, 4 \right\}},\ w \in W \setminus \left\{0 \right\}\right\}},
\end{displaymath}
so again we have
\begin{displaymath}
    G^{\left(2 \right)} \leqslant \Aut{\left(\Gamma_2 \right)} \leqslant \left(V \otimes W \right) \rtimes {\left(\GL{\left(2, 13 \right)} \circ \GL{\left(m, 13 \right)}\right)}.
\end{displaymath}
Using
\begin{align*}
    V_2 &= \left(\left<e_1 + 2 e_2 \right> \cup \left<e_1 + 6 e_2 \right> \cup \left<e_1 + 7 e_2 \right> \cup \left<e_1 + 11 e_2 \right> \right) \setminus {\left\{0 \right\}}, \\
    V_3 &= \left(\left<e_1 + 3 e_2 \right> \cup \left<e_1 + 4 e_2 \right> \cup \left<e_1 + 9 e_2 \right> \cup \left<e_1 + 10 e_2 \right> \right) \setminus {\left\{0 \right\}},
\end{align*}
the result follows using the same argument as the previous two cases.
\end{proof}

Consider now the group $G = G{\left(m, 17 \right)}$. The following result will be used later.

\begin{theorem} \label{thm:q=17}
The group $G{\left(m, 17 \right)}$ is the automorphism group of its orbital digraph union $\Gamma = \Gamma_1 \cup \Gamma_2$ for all $m \geqslant 2$.
\end{theorem}
\begin{proof}
Letting $\Delta = \Delta_1 \cup \Delta_2$, we have $\Gamma = \Cay{\left(V \otimes W,\ \Delta \right)}$. If we define $\mu_1 = 1$, $\mu_2 = 2$, $\mu_3 = 8$, $\mu_4 = 9$, $\mu_5 = 15$ and $\mu_6 = 16$, then
\begin{displaymath}
    \Delta = {\left\{\left(e_1 + \mu_i e_2 \right) \otimes w \mid i \in {\left\{1, 2, 3, 4, 5, 6 \right\}},\ w \in W \setminus \left\{0 \right\}\right\}},
\end{displaymath}
so Proposition~\ref{prop:automorphism-group-affine-decomp} asserts that $\Aut{\left(\Gamma \right)} \leqslant \left(V \otimes W \right) \rtimes \left(\GL{\left(2, 17 \right)} \circ \GL{\left(m, 17 \right)}\right)$. Letting $A \circ B \in \Aut{\left(\Gamma \right)}_0$, it suffices to show that $A \circ B \in G_0$. We know from Lemma~\ref{lm:endomorphism-automorphism} that $A \circ B$ preserves $\Delta$. It follows that $A$ must preserve the set
\begin{displaymath}
    V_{1,2} = \bigcup_{i=1}^6 \left<e_1 + \mu_i e_2 \right> \setminus {\left\{0 \right\}}.
\end{displaymath}
It can be checked computationally that the stabiliser of $V_{1,2}$ in $\GL{\left(2, 17 \right)}$ is precisely $D_8$, so $A \circ B$ lies in $G_0$ as required.
\end{proof}

\section{Fields of prime order}

In this section, let $G = G{\left(m, p \right)}$ for some prime $p \geqslant 5$, and suppose that $\Aut{\left(\Gamma_\lambda \right)} > G$ for some $\lambda \in \GF{\left(p \right)}$ with $\lambda^4 \notin \left\{0, 1 \right\}$. Then there must exist some $\theta \in \Aut{\left(\Gamma_\lambda \right)} \setminus G$ that fixes the zero vector, and from Proposition~\ref{prop:automorphism-group-affine-decomp} we have that $\theta = A \circ B$ for some $A \in \GL{\left(2, p \right)}$ and $B \in \GL{\left(m, p \right)}$. Since $\theta$ preserves $\Delta_\lambda$ by Lemma~\ref{lm:endomorphism-automorphism}, it follows that $A$ must preserve the set
\begin{displaymath}
V_\lambda = \left<\left(1,\ \lambda \right)\right> \cup \left<\left(1,\ -\lambda \right)\right> \cup \left<\left(1,\ \lambda^{-1} \right)\right> \cup \left<\left(1,\ -\lambda^{-1} \right)\right> \subseteq V,
\end{displaymath}
where the vectors are written with respect to the basis $\left\{e_1, e_2 \right\}$.

Since $A$ is a linear transformation of $V$, it must permute the one-spaces $\left<\left(1,\ \lambda \right)\right>$, $\left<\left(1,\ -\lambda \right)\right>$, $\left<\left(1,\ \lambda^{-1} \right)\right>$ and $\left<\left(1,\ -\lambda^{-1} \right)\right>$. If we label these one-spaces respectively as $P$, $Q$, $R$ and $S$, then we can say $A$ induces some permutation $\sigma \in \Sym(\left\{P, Q, R, S \right\})$.

Suppose, for contradiction, that $\sigma$ lies in the permutation group $V_4$, defined as
\begin{displaymath}
V_4 = {\left\{(),\ {\left(PQ \right)}{\left(RS \right)},\ {\left(PR \right)}{\left(QS \right)},\ {\left(PS \right)}{\left(QR \right)} \right\}}.
\end{displaymath}
The one-spaces $P$, $Q$, $R$ and $S$ can be considered as points in the projective space $\PG{\left(1, p \right)}$. Since $\PGL{\left(2, p \right)}$ acts sharply $3$-transitively on $\PG{\left(1, p \right)}$, no two collineations in $\PGL{\left(2, p \right)}$ induce the same permutation of the points $\left\{P, Q, R, S \right\}$.

Define the $2 \times 2$ matrix $M$ as in Table~\ref{tbl:V_4-collineations}. For each of the four possibilities for $\sigma$, it is easy to check that the collineation $M$ induces the permutation $\sigma$. Moreover, observe that $M \in D_8$ for any $\sigma \in V_4$.

\begin{table}
\centering
\begin{tabular}{cc}
\hline
$\sigma$ & $M$ \\
\hline
$()$ & $\begin{bmatrix} 1 & 0 \\ 0 & 1 \end{bmatrix}$ \\
${\left(PQ \right)}{\left(RS \right)}$ & $\begin{bmatrix} 1 & 0 \\ 0 & -1 \end{bmatrix}$ \\
${\left(PR \right)}{\left(QS \right)}$ & $\begin{bmatrix} 0 & 1 \\ 1 & 0 \end{bmatrix}$ \\
${\left(PS \right)}{\left(QR \right)}$ & $\begin{bmatrix} 0 & -1 \\ 1 & 0 \end{bmatrix}$ \\
\hline
\end{tabular}
\caption{Collineations $M$ that permute $\left\{P, Q, R, S \right\}$ by the permutation $\sigma \in V_4$.}
\label{tbl:V_4-collineations}
\end{table}

Since the collineation $A$ also induces the permutation $\sigma$, we must have $A = kM$ for some $k \in \GF{\left(p \right)}$. But then
\begin{displaymath}
\theta = \left(kM \right) \circ B = M \circ \left(kB \right) \in D_8 \circ \GL{\left(m, p \right)} \leqslant G,
\end{displaymath}
which contradicts our requirement that $\theta \notin G$. Hence $\sigma \notin V_4$.

Consider now the cross-ratio $r = \mathcal{R}{\left(P, Q; R, S \right)}$ of our four projective points. Choosing the reference vectors $u = \left(1, 0 \right)$ and $v = \left(0, 1 \right)$, we find the cross-ratio to be
\begin{align*}
r = \mathcal{R}{\left(P, Q; R, S \right)} &= \mathcal{R}{\left(P_{\left(\lambda \right)}, P_{\left(-\lambda \right)}; P_{\left(\lambda^{-1} \right)}, P_{\left(-\lambda^{-1} \right)}\right)} \\
&= \frac{\left(\lambda^{-1} - \lambda \right) \left(-\lambda^{-1} + \lambda \right)}{\left(\lambda^{-1} + \lambda \right) \left(-\lambda^{-1} - \lambda \right)} \\
&= \frac{\left(\lambda^2 - 1 \right)^2}{\left(\lambda^2 + 1 \right)^2}.
\end{align*}

Since $\sigma \notin V_4$, Lemma~\ref{lm:permute-cross-ratio} gives the following restriction on the cross-ratio $r'$ of the permuted points:
\begin{align} \label{eqn:cross-ratio-possibilities}
r' = \mathcal{R}{\left(P^\sigma, Q^\sigma; R^\sigma, S^\sigma \right)} \in {\left\{\frac{1}{r},\ 1-r,\ \frac{r}{r-1},\ \frac{1}{1-r},\ \frac{r-1}{r} \right\}}.
\end{align}

Moreover, $A$ must preserve the cross-ratio of our four points as it is a collineation of $\PG{\left(1, p \right)}$. Since $A$ induces the permutation $\sigma$, we have
\begin{align} \label{eqn:cross-ratio-preserved}
r' = \mathcal{R}{\left(P^\sigma, Q^\sigma; R^\sigma, S^\sigma \right)} = \mathcal{R}{\left(P, Q; R, S \right)} = r.
\end{align}

From (\ref{eqn:cross-ratio-possibilities}) and (\ref{eqn:cross-ratio-preserved}), we have the following five cases to examine: \\\\
Case 1: $r = \frac{1}{r}$. Then
\begin{displaymath}
1 - r^2 = 0 \implies \frac{8 \lambda^2 \left(\lambda^4 + 1 \right)}{\left(\lambda^2 + 1 \right)^4} = 0 \implies \lambda^4 + 1 = 0.
\end{displaymath}
Case 2: $r = 1-r$. Then
\begin{displaymath}
2r - 1 = 0 \implies \frac{\lambda^4 - 6 \lambda^2 + 1}{\left(\lambda^2 + 1 \right)^2} = 0 \implies \lambda^4 - 6 \lambda^2 + 1 = 0.
\end{displaymath}
Case 3: $r = \frac{r}{r-1}$. Then
\begin{align*}
2r - r^2 = 0 &\implies \frac{\left(\lambda^4 + 6 \lambda^2 + 1 \right) \left(\lambda + 1 \right)^2 \left(\lambda - 1 \right)^2}{\left(\lambda^2 + 1 \right)^4} = 0 \\ &\implies \lambda^4 + 6 \lambda^2 + 1 = 0.
\end{align*}
Case 4: $r = \frac{1}{1-r}$. Then
\begin{displaymath}
r^2 - r + 1 = 0 \implies \frac{\lambda^8 + 14 \lambda^4 + 1}{\left(\lambda^2 + 1 \right)^4} = 0 \implies \lambda^8 + 14 \lambda^4 + 1 = 0.
\end{displaymath}
Case 5: $r = \frac{r-1}{r}$. This is in fact the same equation as in Case 4.

The result we have just proved is summarised in the following lemma.
\begin{lemma} \label{lm:Lambda_lambda}
Let $\lambda \in \GF{\left(p \right)}$ with $\lambda^4 \notin \left\{0, 1 \right\}$. If $\Aut{\left(\Gamma_\lambda \right)} \neq G$, then
\begin{displaymath}
0 \in {\left\{\lambda^4 + 1,\ \lambda^4 \pm 6 \lambda^2 + 1,\ \lambda^8 + 14 \lambda^4 + 1 \right\}}.
\end{displaymath}
\end{lemma}

This leads us to our final result.

\begin{theorem} \label{thm:q-large-prime}
Let $p$ be an odd prime and let $m \geqslant 2$. If $p \notin \left\{3,\, 5,\, 7,\, 13 \right\}$ then $G = G{\left(m, p \right)}$ is the automorphism group of a digraph.
\end{theorem}
\begin{proof}
We have seen in Theorem~\ref{thm:q=17} that $G$ is the automorphism group of a digraph if $p = 17$, so suppose that $p \notin \left\{3, 5, 7, 13, 17 \right\}$ and assume for contradiction that $G$ is not the automorphism group of any digraph. It follows that the orbital digraphs $\Cay{\left(V \otimes W,\ \Delta_2 \right)}$ and $\Cay{\left(V \otimes W,\ \Delta_4 \right)}$ must have automorphisms $\theta_2$ and $\theta_4$ respectively that do not lie in $G$. Since $2^4,\, 4^4 \notin \left\{0, 1 \right\}$, Lemma~\ref{lm:Lambda_lambda} gives us that $0 \in \left\{17,\, 41,\, -7,\, 481 \right\}$ and $0 \in \left\{257,\ 353,\ 161,\ 69121 \right\}$. Equivalently, both of these sets must contain a multiple of $p$. This only holds for $p=7$ and $p=13$, a contradiction.
\end{proof}

\end{document}